\documentclass [12pt,a4paper,reqno]{amsart}
 \usepackage{amsmath,amscd}
 \usepackage{graphicx}
 \makeatletter
\renewcommand*\env@matrix[1][*\c@MaxMatrixCols c]{%
  \hskip -\arraycolsep
  \let\@ifnextchar\new@ifnextchar
  \array{#1}}
\makeatother

\def\dispace{\setlength{\itemsep}{2pt}}

\def\Perp{\operatorname*{\perp}}
\def\bigPerp{\operatorname*{\mbox{\large $\perp$} } }
\def\BigPerp{\operatorname*{\mbox{\Large $\perp$} } }
\def\mfB{\mathfrak B}
\def\mfC{\mathfrak C}
\def\alt{{\operatorname{alt}}}

 \input xy
\xyoption{all}

\newcommand{\etype}[1]{\renewcommand{\labelenumi}{(#1{enumi})}}
\def\eroman{\etype{\roman}}
\def\ealph{\etype{\alph}}

\def\pSkip{\vskip 1.5mm \noindent}
\newcommand{\ds}[1]{\ {#1} \ }
\newcommand{\dss}[1]{\quad {#1} \quad }

\def\sm{\setminus}
\def\00{ \{ 0 \}}

\def\veps{\varepsilon}

\def\X1{X_1}
\def\Y1{Y_1}

\def\tlb{\tilde b}
\def\tlq{\tilde q}
\def\tlV{\widetilde V}

\def\t{{\operatorname {t}}}

\def\al{\alpha}
\def\bt{\beta}
\def\gm{\gamma}
\def\Gm{\Gamma}
\def\tlGm{\widetilde{\Gm}}
\def\sig{\sigma}

\newcommand{\opl}{\operatornamewithlimits{\bigoplus\limits}}

\newtheorem{thm}{Theorem} [section]
\newtheorem*{thm*}{Theorem}
\newtheorem{cor}[thm]{Corollary}
\newtheorem{lem}[thm]{Lemma}
\newtheorem{lemma}[thm]{Lemma}
\newtheorem{prop}[thm]{Proposition}

\newtheorem*{claim*} {Claim}
\newtheorem*{theorem4.5'} {Theorem 4.5$'$}
\newtheorem{acknowledgment*}[thm] {Acknowledgment}
\newtheorem{example}[thm]{Example}
\newtheorem{examp}[thm]{Example}
\newtheorem*{examp*}{Example}

\newtheorem{examples}[thm]{Examples}
 \newtheorem{remark}[thm]{Remark}
  \newtheorem{remarks}[thm]{Remarks}
 \newtheorem*{remark*}{Remark}
 \newtheorem{defn}[thm]{Definition}
\newtheorem{construction}[thm]{Construction}

\newtheorem{schol}[thm]{Scholium}

\newtheorem{notationdefinition}[thm]{Notation/Definition}

\newtheorem*{notation*} {Notation}

\theoremstyle{remark}
\newtheorem*{caution*} {Caution}
\newtheorem*{comment*} {Comment}

\newcommand{\thmref}[1]{Theorem~\ref{#1}}
\newcommand{\propref}[1]{Proposition~\ref{#1}}

\newcommand{\lemref}[1]{Lemma~\ref{#1}}

\def\NQL{$\operatorname{NQL}$}

 \renewcommand{\sectionmark}[1]{}

\newcommand{\bfem}[1]{\textbf{#1}}

 \newcommand{\dl}{\delta}
\newcommand{\Dl}{\Delta}
\newcommand{\lm}{\lambda}
\newcommand{\Lm}{\Lambda}

 \newcommand{\supp} {\operatorname{supp}}

\begin{document}

\title[Quadratic and Symmetric Bilinear Forms on Modules ]
{Quadratic and Symmetric Bilinear Forms \\[2mm] on Modules with Unique Base \\[2mm] Over a Semiring}
\author[Z. Izhakian]{Zur Izhakian}
\address{Institute  of Mathematics,
 University of Aberdeen, AB24 3UE,
Aberdeen,  UK.}
    \email{zzur@abdn.ac.uk; zzur@math.biu.ac.il}
\author[M. Knebusch]{Manfred Knebusch}
\address{Department of Mathematics,
NWF-I Mathematik, Universit\"at Regensburg 93040 Regensburg,
Germany} \email{manfred.knebusch@mathematik.uni-regensburg.de}
\author[L. Rowen]{Louis Rowen}
 \address{Department of Mathematics,
 Bar-Ilan University,  Ramat-Gan 52900, Israel}
 \email{rowen@math.biu.ac.il}






\date{\today}






\begin{abstract} We study quadratic forms on free modules with
unique base, the situation that arises in tropical algebra, and
prove the analog of Witt's Cancellation Theorem. Also, the tensor
product of an indecomposable bilinear module $(U, \gm)$ with an
indecomposable quadratic module $(V,q) $ is indecomposable, with the
exception of one case, where two indecomposable components arise.
\end{abstract}

\maketitle

\tableofcontents

\numberwithin{equation}{section}

\section*{Introduction}

  Recall that a \bfem{semiring}~$R$ is a set~$R$ equipped with addition and multiplication,
  such that both $(R,+)$ and $(R,\cdot)$ are abelian monoids\footnote{A monoid means a semigroup that has a neutral element. Here a semiring $R$ is tacitly assumed to be commutative.} with elements $0=0_R$ and $1=1_R$ respectively, and multiplication distributes over addition in the usual way.
  We always assume that $R$ is a commutative semiring with $1.$
In other words, $R$ satisfies all the properties of a commutative
ring except the existence of negation under addition. We call
  a semiring $R$ a \bfem{semifield}, if every nonzero element of $R$ is invertible;
  hence $R\setminus\{0\}$ is an abelian group.

As in the classical theory,
 one often wants to consider bilinear forms defined on (semi)modules over
  a  semiring $R$, often a supertropical semifield, in order to obtain more sophisticated
  trigonometric information.
A \bfem{module} $V$ over $R$ is an abelian monoid $(V,+)$ equipped
with a scalar multiplication $R\times V\to V,$ $(a,v)\mapsto av,$
such that exactly the same axioms hold as customary for modules if
$R$ is a ring: $a_1(bv)=(a_1 b)v,$ $a_1(v+w)=a_1v+a_1w,$
$(a_1+a_2)v=a_1v+a_2v,$ $1_R\cdot v=v,$ $0_R\cdot v=0_V=a_1\cdot
0_V$ for all $a_1,a_2,b\in R,$ $v,w\in V.$ Then  bilinear forms on
$V$ are defined in the obvious way. We write $0$ for both $0_V$
and~$0_R$, and 1 for $1_R.$

The object of this paper is classify quadratic forms over free
modules, with application to the supertropical setting. Actually,
when considering modules over semifields, one encounters several
versions of ``free,'' as studied in depth in \cite[\S4 and
\S5.3]{IzhakianKnebuschRowen2010LinearAlg}. Here we take the most
restrictive version, and call an $R$-module $V$ \bfem{free}, if
there exists a family $(\veps_i \ds|i\in I)$ in $V$ such that every
$x\in V$ has a unique presentation $x=\sum\limits_{i\in I}
x_i\veps_i$ with scalars $x_i\in R$ and only finitely many $x_i $
nonzero, and we call $(\veps_i \ds |i\in I)$ a \bfem{base} of the
$R$-module $V.$ The obvious example is $V = R^{n}$, with the
classical  base. In fact, any free module with a base of $n$
elements is clearly isomorphic to~$R^{n}$, under the map
$\sum\limits_{i=1}^n x_i\veps_i\mapsto (x_1, \dots, x_n).$ The
results are decisive when $R$ is a so-called \textbf{supertropical
semiring},  cf.~ Theorem~\ref{thm:1.2} below.

%
%
%

As in the classical theory, we are led naturally to our main notion
of this paper: For the reader's convenience, we quote some
terminology and results from \cite[\S1-\S4]{QF1}.

 \begin{defn}\label{defn:I.0.1} For any   module $V$  over a semiring $R$,
    a \bfem{quadratic form
on} $V$ is a function $q: V\to R$
with\begin{equation}\label{eq:I.0.1} q(ax)=a^2q(x)\end{equation}
 for any $a\in R,$ $x\in V,$ together with a
symmetric bilinear form $b: V\times V\to R$ (not necessarily
uniquely determined by $q$) such that for any $ x,y\in V$
\begin{equation}\label{eq:I.0.2} q(x+y)=q(x)+q(y)+b(x,y).\end{equation}
 Every such bilinear form $b$ will be called a
\bfem{companion} of $q$, and the pair $(q,b)$ will be called a
\bfem{quadratic pair} on $V.$ We also call $V$ a \textbf{quadratic
module}.
\end{defn}
 When $R$ is a
ring, then $q$ has just one companion, namely,
$b(x,y):=q(x+y)-q(x)-q(y),$ but if $R$ is a semiring that cannot be
embedded into a ring, this usually is not the case, and it is a
major concern of quadratic form theory over semiring to determine
all companions of a given quadratic form $q: V \to R$. Much of the
paper \cite{QF1} is devoted to this problem in the case that $V$ is
a free $R$-module over  a  supertropical semiring, but the first
four sections of~\cite{QF1} deal with quadratic forms and pairs over
an arbitrary semiring, and we draw from these  results in the
present paper.

A quadratic form $q: V \to R$ is called \textbf{quasilinear}, if $q$
has the companion $b =0 $, i.e., $q(x+y) = q(x) + q(y)$ for all $x,y
\in V.$ Assume that $V$ is free with base $(\veps_i \ds: i \in I).$
Then quasilinearity of $V$ implies that, for any vector $x=
\sum\limits_{i\in I } x_i\veps_i$ in $V$,
\begin{equation}\label{eq:I.0.3} q(x)=\sum\limits_{i=1}^n x_i^2 q(\veps_i),\end{equation}
i.e., $q$ has \textbf{diagonal form} with respect to the base
$(\veps_i \ds: i \in I)$.

Under the assumption that for all $a,b \in R$
\begin{equation}\label{eq:I.0.4} (a+b)^2 = a^2 + b^2,\end{equation}
we read off from \eqref{eq:I.0.2} that all diagonal forms are
quasilinear; so diagonality means the same as quasilinearity. In the
present paper we seldom require that $R$ has property
\eqref{eq:I.0.4}, but ``\emph{partial quasilinearity},'' defined as
follows,  plays a major role when we consider orthogonal
decompositions of quadratic modules.

 \begin{defn}\label{defn:I.0.2}

 Given subsets $S$ and $T$ of  $V$, we say that $q$ is \bfem{quasilinear on}
 $S \times T$ if
\begin{equation*} q(x + y)= q(x) + q(y).\end{equation*}
for all $ x\in S$, $y \in T.$
  \end{defn}
The following fact will be of help below, as special case of
\cite[Lemma 1.8]{QF1}. (We write $S+ S'$ for $ \{ s +s' \ds: s\in S,
s' \in S' \}.$) \begin{lem}\label{lem:I.0.3} Let $S, S', T$ be
subsets of $V$. If $q$ is quasilinear on $S \times T$, $S' \times T$
and $S \times S'$, then $q$ is quasilinear on $(S+ S') \times T$.
  \end{lem}

On the other hand, a quadratic form $q: V \to R$ is called
\textbf{rigid} if $q$ has only one companion. For $V$  free with
base $(\veps_i \ds: i \in I),$   $q$ is rigid whenever $q(\veps_i)
=0$ for all $i \in I$ \cite[Proposition~3.4]{QF1}. Under the
assumption that $R$ has property \eqref{eq:I.0.4} and that, for any
$a\in R,$
 \begin{equation}\label{eq:I.0.5} a+ a = 0 \dss{\Rightarrow} a = 0\end{equation}
the converse holds too,
 so by  \cite[Theorem 3.5]{QF1} the rigid forms are precisely those with
$q(\veps_i) = 0$ for all $i \in I$. \{We note in passing that both
\eqref{eq:I.0.4} and \eqref{eq:I.0.5} are valid when $R$ is
supertropical.\}

We say that  $V$ is an \textbf{$R$-module with unique base},  if $V$
is a  free $R$-module and, given a base $\mfB = \{ \veps_i \ds: i
\in I\}$ of $V$, any other base of $V$ is obtained from $ \mfB$ by
multiplying the $\veps_i$ by units of $R$. In the most important
case that $I$ is finite, i.e., $I = \{1, \dots, n \}$, we have:

\begin{remark}\label{chbas} Any change of base of the free module $R^{n}$ is attained by multiplication by an
invertible $n\times n$ matrix, so having unique base is equivalent
to  every invertible matrix  in $M_n(R)$ being a generalized
permutation matrix.\end{remark}

The present paper is devoted to a study of the
quadratic  and symmetric bilinear forms on $R$-modules with unique base.
More specifically we work on  orthogonal decompositions of such forms,
and  on tensor products of two symmetric bilinear forms and
of  a symmetric bilinear form with a quadratic form.

So our first question is, ``What conditions on the semiring $R$
guarantee that  $R^{n}$ has unique base, or equivalently, that every
invertible matrix is generalized permutation?'' The matrix question
was answered by \cite{Tan, Dol}. In their terminology, an
``antiring'' is a semiring~$R$ such that $R \sm \{0\}$ is closed
under addition, and they classify the invertible matrices over
antirings. These are just the generalized permutation matrices when
$R\setminus \{ 0 \}$ also is closed under multiplication, which they
call ``entire'' (the case in tropical mathematics), and more
generally by \cite[Theorem~1]{Dol} (as interpreted in
Theorem~\ref{thm:1.7}) when $R$ is indecomposable, i.e., not
isomorphic to a direct product $R_1 \times R_2$ of semirings.

In our proofs of all of our results, we never use the matrix
interpretation of the unique base property. \{In fact matrices show
up only once, in Corollary~\ref{cor:4.6}.\}

 Other than the trivial fact that every free
$R$-module of rank 1 has unique base, all examples known to us of
modules with unique base  emanate from  Theorem~\ref{thm:1.7}. But
we feel that many arguments and related problems left open in the
paper are clearer if, when possible, we assume only that the
$R$-modules considered have unique base. (Only in  \S\ref{sec:6} do
we deviate from this strategy.)

%

%

In \S\ref{sec:2} we develop the notion of \textbf{(disjoint)
orthogonality} of two given disjoint submodules $W_1$ and~$W_2$ of a
\textbf{quadratic $R$-module} $(V,q)$ (endowed with a fixed
quadratic form $q$), which means that $q$ is  partially quasilinear
on $W_1 \times W_2$. (Note that there is no direct reference to an
underlying symmetric bilinear form.) When $V$ has unique base, we
look for orthogonal decompositions $V = W_1 \perp W_2$, more
generally $V = \bigPerp\limits_{i\in I} W_i$, where the $W_i$ are
\textbf{basic submodules}  of $V$, i.e., are generated by subsets of
a base $\mfB$ of $V$.

Theorem~\ref{thm:2.6} shows that there is a unique disjoint
orthogonal decomposition of $V$ into indecomposable basic
submodules.

In \S\ref{sec:3} we develop the analogous notion of disjoint
orthogonality in a \textbf{bilinear $R$-module} $(V,b)$ with respect
to  a fixed symmetric bilinear  form $b$ on $V$, and by passing from
$q$ to a suitable companion, we obtain the same
indecomposable basic submodules (Theorem~\ref{thm:3.9}).
 Here it helps to
modify $b$ to a related symmetric bilinear form $b_\alt$
(Definition~\ref{defn:3.5}) which facilitates a description of
$(V,b)$ as spanned by the connected components of a graph associated
with the base.


In \S\ref{sec:4}, these decomposition theories yield  an analogue of
Witt's cancellation theorem over fields of characteristic $\neq 2$
\cite{Witt}, given as Theorem~\ref{thm:4.9}: 
 If $W_1, W_1',
W_2, W_2'$ are finitely generated quadratic or bilinear modules with
unique base such that $W_1 \cong W_1'$ and $W_1 \perp W_2 \cong W_2
\perp W_2'$, then $W_2 \cong W_2'$ (where $\cong$ means
``isometric'').
Theorem~\ref{thm:4.9}  vindicates our somewhat exotic notion of
disjoint orthogonality. It actually  is given in more general terms,
where  $W_2$ need not be  finitely generated.

\pSkip 
The last two sections of the paper are devoted to tensor products of
arbitrary $R$-modules over a semiring $R$. While the theory of
tensor products of $R$-modules over general semirings can be carried
out in analogy to the usual classical construction over rings, it
requires the use of congruences, resulting in some technical issues
dealt with in \cite{H}, for example. But for free modules the basics
can be done as easily as over rings, especially when the bases are
assumed to be unique (since then one does not need to worry about
well-definedness).

In \S\ref{sec:5} we construct the \textbf{tensor product of two free
bilinear $R$-modules} over any semiring $R$, analogous to the case
where $R$ is a ring, cf.~\cite[\S2]{EKM}, \cite[I, \S5]{Spez}. We
then take the \textbf{tensor product of a free bilinear $R$-module
$U = (U, \gm)$ with a free quadratic $R$-module $V= (V,q)$}. A new
phenomenon occurs  here, in contrast to the theory over rings.  It
is necessary first to choose  a so called \textbf{balanced
companion} $b$ of $q$, which always exists, cf.~\cite[\S1]{QF1}, but
which is usually not unique. We then define the tensor product $U
\otimes_b  V$, depending on $b$, by choosing a so called
\textbf{expansion} $B: V \times V \to R$ of the quadratic pair
$(q,b)$ which is a (often non-symmetric) bilinear form $B$ with
$$ B(x,x) = q(x), \qquad B(x,y) + B(y,x) = b(x,y)$$
for all $x,y \in V$, cf.~\cite[\S1]{QF1} and then proceeding
essentially as in the case of rings, e.g.~\cite[Definition
1.51]{Spez}, \cite[p. 51]{EKM}\footnote{For $R$ a ring the ``$b$''
in the tensor product does not need to be specified since $q$ has
only one companion.}. The resulting quadratic form $\gm \otimes_b q$
does not depend on the choice of $B$ but often depends on the choice
of $b$. This is apparent already in the case $\gm = \left(
\begin{smallmatrix}
                                              0 & 1 \\
                                              1 & 0
                                            \end{smallmatrix}\right)$, where the matrix $b$ is stored in the quadratic polynomial $\gm \otimes_b q$, cf.~Example~ \ref{examp:5.8} below.

 In \S\ref{sec:6} we determine
  the indecomposable components of tensor products of  modules with unique base, first of two
indecomposable bilinear (free) $R$-modules, and then of an
indecomposable bilinear $R$-module with an indecomposable quadratic
$R$-module. For simplicity we assume here that $R \sm \{ 0\}$ is an
entire antiring, i.e., closed under multiplication and addition,
relying on Theorem \ref{thm:1.3} that, in this case, all free
$R$-modules have unique base.

Our main result of this section, Theorem \ref{thm:6.16}, states
that, discarding trivial situations and excluding some pathological
semirings, the tensor product of an indecomposable bilinear module
$(U, \gm)$ with an indecomposable quadratic module $(V,q) $ is again
indecomposable, with the exception of one case, where two
indecomposable components arise.

The proof of this result, and of the preceding theorems
\ref{thm:6.6} and \ref{thm:6.8} about tensor products of
indecomposable  bilinear modules as well, has a graph theoretic
flavor. We work with ``paths'' and ``cycles'' in the bases of $U,V$
and $U \otimes_R V$, and indeed we could associate graphs to $(U,
\gm), $ $(V,q)$, and $ (U,\gm) \otimes_b (V,q)$ in a way obvious
from the arguments, where these paths and cycles get the usual
graph-theoretic meaning. We have refrained from appealing to graph
theory here, since at this stage  no deeper theorems about graphs
are needed.

\section{$R$-modules with unique base and their basic submodules}\label{sec:1}

 \begin{defn}\label{defn:1.1}
An \textbf{$R$-module with unique base} is a free $R$-module $V$  in
which any two bases $\mfB,$ $\mfB'$ are projectively the same, i.e.,
we obtain the elements of $\mfB'$ from those of $\mfB$ by
multiplying by units of $R.$
\end{defn}

Our interest in these modules originates from the following two key facts.

\begin{thm}\label{thm:1.2} {\rm(cf.~\cite[Proposition~0.9]{QF1})} If $R$ is a supertropical semiring, then
every free $R$-module has unique base.
\end{thm}

\begin{thm}\label{thm:1.3} (cf.~\cite[\S2, Corollary 3]{Dol}, an alternative proof below.) If the set $R\setminus \{0\}$ is closed under addition
and multiplication (i.e., $a+b=0\Rightarrow a=b=0,$ $a\cdot
b=0\Rightarrow a=0$ or $b=0$), then every free $R$-module has
unique base.
\end{thm}

 Assume now that $V$ is
 a free $R$-module and $\mfB$ is a fixed base of $V.$

\begin{defn}\label{defn:1.4}
We call a submodule $W$ of $V$ \bfem{basic}, if $W$ is spanned by
$\mfB_W:=\mfB\cap W,$ and thus $W$ is free with base $\mfB_W.$ Note
that then we have a unique direct decomposition $V=W\oplus U,$ where
the submodule $U$ is basic with base $\mfB\setminus\mfB_W.$ $W$ and
$U$  again are $R$-modules with unique base. We call $U$ the
\bfem{complement of} $W$ \bfem{in} $V,$ and write $U=W^c.$
\end{defn}

The theory of basic submodules of $V$ is of utmost simplicity. All
of the following is obvious.

\begin{schol}\label{schol:1.5} $ $
\begin{enumerate} \ealph

\item We have a bijection $W\mapsto \mfB_W:=\mfB\cap W$ from the set of basic submodules of
$V$ onto the set of subsets of $\mfB.$ \pSkip

\item If $W_1$ and $W_2$ are basic submodules of $V,$ then also $W_1\cap W_2$ and $W_1+W_2$ are  basic submodules of  $V,$ and
$$\mfB_{W_1\cap W_2}=\mfB_{W_1}\cap \mfB_{W_2},\qquad \mfB_{W_1+W_2}=\mfB_{W_1}\cup \mfB
_{W_2}.$$ \pSkip

\item If $W$ is a basic submodule of $V,$ then as stated above,
$$\mfB_{W^c}=\mfB\setminus \mfB_W.$$
\pSkip

\item Finally, if $W_1\subset W_2$ are basic submodules of $V,$ then $W_1$ is basic in $W_2$ and
$W_1^c\cap W_2$ is the complement of $W_1$ in $W_2.$
\end{enumerate}
\end{schol}

In view of Remark~\ref{chbas}, \thmref{thm:1.3} follows from Dol\u
zan and Oblak \cite[\S2, Corollary~3]{Dol} using matrix arguments
within a wider context extending work of Tan \cite[Proposition~
3.2]{Tan}, which in turn relies on Golan's  book on semirings
\cite[Lemma 19.4]{golan92}. We now reprove Theorem \ref{thm:1.3} by
a simple matrix-free argument in preparation for a reproof of the
more general Theorem \ref{thm:1.7}.

\begin{proof}[Proof of \thmref{thm:1.3}]
Let $V$ be a free $R$-module and $\mfB$ a base of $V.$ If $x\in V\setminus \{0\}$ is given, we have a
presentation
$$x=\sum_{x=1}^r \lm_ix_i$$
with $x_i\in\mfB$ and $\lm_i\in R\setminus\{0\}$. We call the set
$\{x_1,\dots,x_r\}\subset\mfB$ the \bfem{support} of $x$ with
respect to $\mfB$ and denote this set by $\supp_{\mfB}(x).$ Note
that if $x,y\in V\setminus\{0\},$ then $x+y\ne0$ and
\begin{equation}\label{eq:1.01}
\supp_{\mfB}(x+y)=\supp_{\mfB}(x)\cup\supp_{\mfB}(y)\end{equation}
due to the assumption that $\lm+\mu\ne 0$ for any $\lm,\mu\in
R\setminus\{0\}$. Also
\begin{equation}\label{eq:1.02}
\supp_{\mfB}(\lm x)=\supp_{\mfB}(x)\end{equation}
for $x\in V\setminus\{0\},$ $\lm\in R\setminus\{0\}$, due to the assumption that for $\lm,\mu\in R
\setminus\{0\}$ we have $\lm\mu\ne0.$

Now assume that $\mfB'$ is a second base of $V.$ Given $x\in\mfB,$ we have a presentation
$$x=\lm_1y_1+\cdots+\lm_ry_r$$
with $\lm_i\in R\setminus\{0\}$ and distinct $y_i\in\mfB'.$ It
follows from \eqref{eq:1.01} and \eqref{eq:1.02}
that
$$\{x\}=\supp_{\mfB}(x)=\supp_{\mfB}(y_1)\cup\dots\cup\supp_{\mfB}(y_r).$$

This forces
\begin{equation}\label{eq:1.03}\{x\}=\supp_{\mfB}(y_1)=\dots =\supp_{\mfB}(y_r).\end{equation}
From this, we infer that $r=1.$
Indeed, suppose that $r\ge 2.$ Then $y_1=\mu_1x,$ $y_2=\mu_2x$ with $\mu_1,\mu_2\in R\setminus\{0\}.$
But this implies $\mu_2y_1=\mu_1y_2,$ a contradiction since $y_1,y_2$ are different elements of a base of $V.$

Thus $\{x\}=\supp_{\mfB}(y)$ for a unique $y\in\mfB',$ which means $y=\lm x$ with $\lm\in R\setminus\{0\}.$ By symmetry we have a unique $z\in\mfB$ and $\mu\in R\setminus\{0\}$ with $x=\mu z.$ Then $x=\lm\mu z,$ whence $x=z$ and $\lm\mu=1.$ Thus $\lm,\mu\in R^*$ and $x\in R^*y,$ $y\in R^*x.$ Of course, $y$ runs through all of ~$\mfB'$ if $x$ runs through $\mfB,$ since both $\mfB$ and $\mfB'$ span the module $V.$
\end{proof}

With further effort, we  now obtain a theorem that encompasses both
Theorems \ref{thm:1.2} and~\ref{thm:1.3}.

\begin{defn}\label{defn:1.6} We say that the semiring $R$ is \bfem{indecomposable} if $R$ is not isomorphic to a direct product $R_1\times R_2$ of non-zero semirings $R_1$ and $R_2;$ in other words, there do not exist idempotents $\mu_1\ne0$ and $\mu_2\ne0$ in $R$ with $\mu_1\mu_2=0$ and $\mu_1+\mu_2=1.$
\end{defn}

\begin{thm}[{\cite[Theorem~1]{Dol}}]\label{thm:1.7} Assume that $R\setminus\{0\}$ is an indecomposable antiring. Then every free $R$-module
has unique base.
\end{thm}

\begin{proof}
     Assume that $\mfB$ and $\mfB'$ are  bases of $V.$
     Given $x\in V\setminus\{0\}$, we write again
\begin{equation}\label{eq:1.1}
x=\lm_1y_1+\dots+\lm_ry_r
\end{equation}
with different $y_i\in\mfB',$ $\lm_i\in R \sm \{0\}.$ But now, instead of \eqref{eq:1.03} we can only conclude  that
\begin{equation}\label{eq:1.2.0}
\{x\}=\supp_{\mfB}(\lm_1y_1)=\dots=\supp_{\mfB}(\lm_iy_i).\end{equation}
 Thus we have scalars $\mu_i\in R\setminus\{0\}$ such that
\begin{equation}\label{eq:1.2}
 \lm_iy_i=\mu_ix\qquad\text{for}\quad 1\le i\le r.
\end{equation}
Suppose that $r\ge2.$ Then we have for all $i,j\in\{1,\dots, r\}$ with $i\ne j.$
$$\mu_j\lm_iy_i=\mu_j\mu_ix=\mu_i\mu_jx=\mu_i\lm_jy_j.$$
Since the $y_i$ are elements of a base, this implies $\mu_i\lm_j=\mu_j\lm_i=0$ for $i\ne j$ and then
\begin{equation}\label{eq:1.3}
\mu_i\mu_j=0\qquad\text{for}\quad i\ne j.
\end{equation}
On the other hand, we obtain from \eqref{eq:1.1} and \eqref{eq:1.2} that
$$x=\mu_1x+\mu_2x+\dots +\mu_rx,$$
and then
\begin{equation}\label{eq:1.4}
1=\mu_1+\mu_2+\dots +\mu_r.
\end{equation}
Multiplying \eqref{eq:1.4} with $\mu_i$ and using  \eqref{eq:1.3}, we obtain
\begin{equation}\label{eq:1.5}
 \mu_i^2=\mu_i.
\end{equation}
Thus
$$R\cong R\mu_1\times\dots\times R\mu_r.$$
This contradicts our assumption that $R$ is indecomposable.

We have proved that $r=1.$ Thus for every $x\in\mfB$ there exist unique $y\in\mfB'$ and $\lm\in R$ with $x=\lm y.$ By the same argument as in the end of proof of \thmref{thm:1.3}, we conclude that~$\mfB$ is projectively unique.
\end{proof}

 Of course, if $R\setminus\{0\}$ is closed under multiplication, i.e., $R$ has no zero divisors, then $R$ is indecomposable. This also holds when $R$ is supertropical (cf.~\cite[\S3]{IR1}, \cite[Definition 0.3]{QF1}), since then  for any two elements $\mu_1,\mu_2$ of $R$ with $\mu_1+\mu_2=1$ either $\mu_1=1$ or $\mu_2=1.$ Thus \thmref{thm:1.7} generalizes both Theorems  \ref{thm:1.2}  and \ref{thm:1.3}.

 The following example reveals that \thmref{thm:1.7} is the best we can hope for to guarantee
 that every free $R$-module has unique base, as long as we stick to the condition that $R$ is an antiring,
 a natural assumption for the remainder of this paper.

 \begin{examp}\label{examp:1.8}
If $R_0$ is an antiring, then $R:=R_0\times R_0$
  also is an antiring. Put $\mu_1=(1,0),$ $\mu_2=(0,1).$ These are idempotents in $R$ with $\mu_1\mu_2=0$ and $\mu_1+\mu_2=1.$ Now let $V$ be a free $R$-module with base $\mfB=\{\veps_1,\veps_2,\dots, \veps_n\},$ $n\ge 2,$  choose a permutation $\pi\in S_n,$ $ \pi\ne1,$ and define
 $$\veps_i':=\mu_1\veps_i+\mu_2\veps_{\pi(i)}\qquad (1\le i\le n).$$
 We claim that $\mfB ':=\{\veps_1',\dots, \veps_n'\}$ is  another base of $V.$

 Indeed, $V$ is a free $R_0$-module with base $(\mu_i\veps_j\ds|1\le i\le 2,\, 1\le j\le n).$ We have $$\mu_1\veps_i'=\mu_1\veps_i, \qquad \mu_2\veps_i'=\mu_2\veps_{\pi (i)},$$ and thus $(\mu_i\veps_j'\ds|1\le i\le 2,\, 1\le j\le n)$ is a permutation of this base over $R_0,$ i.e., regarded as a set, the same base. Thus certainly $\mfB'$ spans $V$ as  $R$-module.

 Given $x\in V,$ let $x=\sum\limits_1^n a_i\veps_i'$ with $a_i\in R.$ We have
 $$a_i=a_{i1}\mu_1+a_{i2}\mu_2\qquad\text{with}\qquad a_{i1}\in R_0,\ a_{i2}\in R_0,$$
 whence
 $$x=\sum_{i=1}^na_{i1}(\mu_1\veps_i)+\sum_{i=1}^na_{i2}(\mu_2\veps_{\pi(i)}).$$
 This shows that the coefficients $a_{i1},a_{i2}\in R_0$ are uniquely determined by $x,$ whence the coefficients $a_i\in R$ are also uniquely determined by $x.$ Our claim is proved.

 Since $\supp_{\mfB}(\veps_i')$ has two elements if $\pi(i)\ne i,$ $\mfB'$
 differs projectively from $\mfB.$ The base of the $R$-module~$V$ is not   unique.
 \end{examp}

 \section{Orthogonal decompositions of quadratic modules with unique base}\label{sec:2}

 Assume that $V$ is an $R$-module equipped with a fixed quadratic form $q:V\to R$. We then call $V=(V,q)$ a \bfem{quadratic $R$-module}.

 \begin{defn}\label{defn:2.1}  $ $
 \begin{enumerate} \ealph
   \item  Given two submodules $W_1,W_2$ of the $R$-module $V,$ we say that $W_1$ is \bfem{disjointly orthogonal}\footnote{Later we  say  ``orthogonal" for short, instead of ``disjointly orthogonal", when
   it is clear  a priori that $W_1\cap W_2=\{0\}.$} to $W_2,$ if $W_1\cap W_2=\{0\}$ and $q(x+y)=q(x)+q(y)$ for all $x\in W_1,$ $y\in W_2,$ i.e., $q$ is quasilinear on $W_1\times W_2.$
\pSkip
  \item We write $V=W_1 \Perp W_2$ if  $V=W_1\oplus W_2$ (as $R$-module)
  with $W_1$   disjointly orthogonal to $W_2 $.
We then call $W_1$ an  \bfem{orthogonal summand} of $W$, and $W_2$
an \bfem{orthogonal complement} of $W_1$ in $V.$
 \end{enumerate}
\end{defn}

\begin{caution*}
If $V =W_1\perp W_2$, we may choose a companion $b$ of $q$ such that
$b(W_1,W_2)=0,$ but note that it could well happen that the set of
all $x\in V$ with $b(x,W_1)=0$ is bigger than $W_2,$ even if $R$ is
a semifield and $q|W_1$ is anisotropic  (e.g., if $q$ itself is
quasilinear). Our notion of orthogonality does not refer to any
bilinear form.
\end{caution*}

We now also define infinite orthogonal sums. This seems to be natural, even if we are originally interested only in
finite orthogonal sums. Indeed, even if $R$ is a semifield, a free
$R$-module with finite base
often has many submodules which are not  finitely generated.

\begin{defn}\label{defn:3.2}
Let $(V_i\ds|i\in I)$ be a family of submodules of the quadratic module $V.$ We say that $V$ is \bfem{the orthogonal sum of the family} $(V_i)$, and then write
$$V= \bigPerp_{i\in I} V_i,$$
if for any two different indices  $i,j$ the submodule $V_i$ is disjointly  orthogonal to $V_j,$ and moreover
 $V=\opl\limits_{i\in I} V_i.$
\end{defn}

N.B. Of course, then for any subset $J\subset I,$ the module
$V_J=\sum\limits_{i\in J}  V_i$ is the orthogonal sum of the
subfamily $(V_i\ds|i\in J);$ in short,
$$V_J=\bigPerp_{i\in J} V_i.$$

\medskip

We state a fact which, perhaps contrary to first glance,  is not completely trivial.

\begin{prop}\label{prop:2.3}
Assume that we are given an orthogonal decomposition $V= \bigPerp\limits_{i\in I} V_i.$ Let~$J$ and $K$ be two disjoint subsets of $I$. Then the submodule $V_J=\bigPerp\limits_{i\in J} V_i$ of $V$ is disjointly orthogonal to $V_K=\bigPerp\limits_{i\in K} V_i$, and thus
$$V_{J\cup K}=V_J \perp V_K.$$
\end{prop}

  \begin{proof}
  It follows from Lemma \ref{lem:I.0.3} above that for any three different indices $i,j,k$
  the form $q$ is quasilinear on $V_i\times (V_j+V_k),$ and thus $V_i$ is orthogonal to $V_j\perp V_k.$
  By iteration, we see that the claim holds if $J$ and $K$ are finite. In the general case, let $x\in V_J$ and $y\in V_K$. There exist finite subsets $J',K'$ of $J$ and $K$ with
$x\in V_{J'},$ $y\in V_{K'},$ and thus $q(x+y)=q(x)+q(y).$ This
proves that $V_J$ is orthogonal to $V_K.$
  \end{proof}

  In the rest of this section, \textbf{we assume that $V$ has unique base}. Then a basic  orthogonal summand~ $W$ of $V$ has only one basic  orthogonal complement, namely, $W^c,$ equipped with the form $q|W^c.$

  \begin{defn}\label{defn:2.4} If the quadratic module $V$ has a basic orthogonal summand $W\ne V$, we call~$V$ \bfem{decomposable}.
  Otherwise we call $V$ \bfem{indecomposable}. More generally, we call a basic submodule $X$ of $V$
  \bfem{decomposable} if $X$ is decomposable with respect to $q|X,$ and otherwise we call $X$ \bfem{indecomposable}.
  \end{defn}

  Our next goal is to decompose the given quadratic module $V$ orthogonally into indecomposable basic submodules. Therefore, we choose a base $\mfB$ of $V$ (unique up to  multiplication by scalar units). We then choose a companion $b$ of $q$
   such that  $b(\veps,\eta)=0$ for any two   \emph{different}
    $\veps,\eta\in\mfB$ such that $q$ is quasilinear on $R\veps\times
    R\eta,$ cf.~\cite[Theorem~6.3]{QF1}.
  We call such a companion $b$ a \bfem{quasiminimal companion} of $q.$

  \begin{comment*} In important cases, e.g.,
  if $R$ is supertropical or more generally ``upper bound" (cf.~\cite[Definition 5.8]{QF1}), the set of companions of $q $ can be partially ordered in a natural way. The prefix ``quasi" here is a reminder that we do not mean minimality with respect to  such an ordering.
  \end{comment*}

  \begin{lem}\label{lem:2.5}
  Let $W$ and $W'$ be basic submodules of $V$ with $W\cap W'=\{0\}.$ If $b$ is any quasiminimal companion of $q,$ then $W$ is (disjointly) orthogonal to $W'$ iff $b(W,W')=0.$
  \end{lem}

  \begin{proof}
  If $b(W,W')=0$, then $q(x+y)=q(x)+q(y)$ for any $x\in W$ and $y\in W'$, which means by definition that $W$ is orthogonal to $W'.$ (This holds for any companion $b$ of $q.$)

  Conversely, if $W$ is orthogonal to $W'$, then for base vectors $\veps\in\mfB_W,$ $\eta\in\mfB_{W'}$ the form $q$ is quasilinear on $R\veps\times R\eta$ and thus $b(\veps,\eta)=0.$ This implies that $b(W,W')=0.$\end{proof}

  We now introduce the following equivalence relation on the set $\mfB$. We choose a quasiminimal companion $b$ of $q.$ Given $\veps,\eta\in\mfB$, we put $\veps\sim\eta,$ iff either $\veps=\eta,$ or there exists a sequence $\veps_0,\veps_1,\dots, \veps_r$ in $\mfB,$ $r\ge 1,$ such that $\veps=\veps_0,$ $\eta=\veps_r,$ and $\veps_i \neq \veps_{i+1}$,  $b(\veps_i,\veps_{i+1})\ne0$ for $i=0,\dots,r-1.$

  \begin{thm}\label{thm:2.6} Let $\{\mfB_k\ds|k\in K\}$ denote the set of equivalence classes in $\mfB $
  and, for every $k\in K,$ let $W_k$ denote the submodule of $V$ having base $\mfB _k.$
  \begin{enumerate}
  \item[{\rm(a)}] Then every $W_k$ is an indecomposable basic submodule of $V$ and
  $$V=\bigPerp_{k\in K} W_k.$$ \pSkip

  \item[{\rm(b)}] Every indecomposable basic submodule $U$ of $V$ is contained in $W_k$, for some $k\in K$ uniquely determined by $U.$ \pSkip
    \item[{\rm(c)}]  The modules $W_k,$ $k\in K,$ are precisely all the indecomposable basic orthogonal summands of $V.$
    \end{enumerate}
      \end{thm}

      \begin{proof} (a): Suppose that $W_k$ has an orthogonal decomposition $W_k=X\perp Y$ with basic submodules $X\ne 0,$ $Y\ne0.$ Then $\mfB_k$ is the disjoint union of the non-empty sets $\mfB_X$ and~ $\mfB_Y.$ Choosing $\veps\in\mfB_X$ and $\eta\in\mfB_Y$, there exists a sequence $\veps_0,\veps_1,\dots,\veps_r$ in $\mfB_k$ with $\veps=\veps_0,$ $\eta=\veps_r$ and $b(\veps_{i-1},\veps_i)\ne0$,   $\veps_{i-1} \neq \veps_{i}$, for $1\le i\le r.$ Let $s$ denote the last index in $\{1,\dots,r\}$ with $\veps_s\in\mfB_X.$ Then $s<r$ and $\veps_{s+1}\in\mfB_Y$. But $b(X,Y)=0$ by \lemref{lem:2.5} and thus $b(\veps_s,\veps_{s+1})=0,$ a contradiction. This proves that $W_k$ is indecomposable. Since $\mfB$ is the disjoint union of the sets $\mfB_k,$ we have
      $$V=\opl_{k\in K} W_k.$$
      Finally, if $k\ne \ell,$ then $b(W_k,W_\ell)=0$ by the nature of our equivalence relation. Thus
      $$V=\BigPerp_{k\in K} W_k.$$ \pSkip

      (b): Given an indecomposable basic submodule $U$ of $V$, we choose $k\in K $ with $\mfB_U\cap\mfB_k\ne \emptyset.$ Then $U\cap W_k\ne0.$ From $V=W_k\oplus W_k^c$, we conclude that $U=(U\cap W_k)\oplus(U\cap W_k^c),$ and then have $U=(U\cap W_k)\perp(U\cap W_k^c)$ because $W_k$ is orthogonal to $W_k^c.$ Since $U$ is indecomposable and $U\cap W_k\ne0,$ it follows that $U=U\cap W_k,$ i.e., $U\subset W_k.$ Since $W_k\cap W_\ell=0$ for $k\ne \ell,$ it is clear that $k$ is uniquely determined by $U.$ \pSkip

      (c): If $U$ is an indecomposable  basic orthogonal summand of $V,$ then $V=U\perp U^c.$ We have $U\subset W_k$ for some $k\in K,$ and obtain $W_k=U\perp (U^c\cap W_k),$ whence $W_k=U.$
      \end{proof}

      \begin{defn}\label{defn:2.6} We call the submodules $W_k$ of $V$ occurring in \thmref{thm:2.6} the \bfem{indecomposable components} of the quadratic module $V.$\end{defn}

      The following facts are easy consequences of the theorem.

      \begin{remark}\label{remarks:2.7}  $ $
      \begin{enumerate} \eroman
        \item
       If $U$ is a basic  orthogonal summand of $V,$ then the indecomposable components of the quadratic module $U=(U,q|U)$ are the indecomposable components of $V$ contained in ~$U.$
\pSkip
     \item  If $U$ is any basic submodule of $V,$ then
      $$U=\BigPerp_{k\in K} \left( U\cap W_k\right),$$
  and every submodule $U\cap W_k\ne\{0\}$ is an orthogonal sum of indecomposable components of $U.$
        \end{enumerate}

  \end{remark}

  \section{Orthogonal decomposition of bilinear modules with unique base}\label{sec:3}

  We now outline a theory of symmetric bilinear forms analogous to the theory for quadratic forms given
  in \S\ref{sec:2}. The bilinear theory is easier than the quadratic theory due the fact that, in contrast to quadratic forms,
    on a free module we do not need to distinguish between ``functional" and ``formal" bilinear forms
   cf.~\cite[\S1]{QF1}. As before, $R$ is a semiring.

  Assume in the following that $V$ is an $R$-module equipped with a fixed symmetric bilinear form $b:V\times V\to R.$ We then call $V=(V,b)$ a \bfem{bilinear $R$-module}.
  If $X$ is a submodule of $V$, we denote the restriction of $b$ to
$X\times X$ by $b|X.$

  \begin{defn}\label{defn:3.1}  $ $
  \begin{enumerate} \ealph
    \item
Given two submodules $W_1,W_2$ of the $R$-module $V$, we say that $W_1$ is \bfem{disjointly orthogonal} to $W_2,$ if $W_1\cap W_2=\{0\}$ and $b(W_1,W_2)=0,$ i.e., $b(x,y)=0$ for all $x\in W_1,$ $y\in W_2.$
\pSkip
 \item We write $V=W_1\perp W_2$ if $W_1$ is disjointly orthogonal to $W_2$ and moreover $V=W_1\oplus W_2$ (as $R$-module). We then call $W_1$ an \bfem{orthogonal summand} of $V$ and $W_2$ an \bfem{orthogonal complement} of $W_1$ in $V.$

  \end{enumerate}
  \end{defn}

\begin{defn}\label{defn:2.2}
Let $(V_i\ds|i\in I)$ be a family of submodules of the bilinear  module $V.$ We say that $V$ is \bfem{the orthogonal sum of the family} $(V_i)$, and then write
$$V= \BigPerp_{i\in I} V_i,$$
if for any two different indices  $i,j$ the submodule $V_i$ is disjointly  orthogonal to $V_j,$ and moreover $V=\opl\limits_{i\in I} V_i.$
\end{defn}

In contrast to the quadratic case, the exact analogue of \propref{prop:2.3} is now a triviality.

\begin{prop}\label{prop:3.3} Assume that $V= \bigPerp\limits_{i\in I} V_i.$ Let $J$ and $K$ be disjoint subsets of $I.$ Then
$V_J=  \bigPerp\limits_{i\in J} V_i $ is disjointly orthogonal to $V_K= \bigPerp\limits_{i\in K} V_i,$ and
$$V_{J\cup K}= V_J\perp V_K.$$
\end{prop}

In the following, \textit{we assume again that $V$ has  unique
base}. Then again a  basic  orthogonal summand~$W$ of $V$ has only
one basic orthogonal complement in $V,$ namely, $W^c$ equipped with
the bilinear form $b|W^c.$

For $X$ a basic submodule of $V$, we define the properties \bfem{``decomposable"} and \bfem{``indecomposable"} in exactly the same way as indicated by Definition \ref{defn:2.4} in the quadratic case.

We start with a definition and description of the ``indecomposable
components" of $V=(V,b)$ in a similar fashion as was done in
\S\ref{sec:2} for quadratic modules. We choose a base $\mfB$ of~ $V$
and again introduce  the appropriate equivalence relation on the set
$\mfB,$ but now we adopt a more elaborate terminology than in
\S\ref{sec:2}. This will turn out to be useful later on.

\begin{defn}\label{defn:3.4}
We call the symmetric bilinear form $b$ \bfem{alternate} if $b(\veps,\veps)=0$ for every $\veps\in\mfB.$
\end{defn}

\begin{comment*}
Beware that this does \bfem{not} imply that $b(x,x)=0$ for every
$x\in V.$ The classical notion of an alternating bilinear form is of
no use here since in the semirings under consideration here
(cf.~\S\ref{sec:1}) $\alpha+\beta=0$ implies $\alpha=\beta=0,$
whence $b(x+y,x+y)=0$ implies $b(x,y)=0.$ An alternating bilinear
form in the classical sense would be identically zero.
\end{comment*}

\begin{defn}\label{defn:3.5}
We associate to the given symmetric bilinear form $b$ an alternate bilinear form $b_{\alt}$ by the rule
$$b_{\alt}(\veps,\eta)=\begin{cases} b(\veps,\eta) &\ \text{if }\ \veps\ne \eta\\
0 & \ \text{if }\ \veps= \eta\end{cases}$$
for any $\veps,\eta\in\mfB.$
\end{defn}

\begin{lem}\label{lem:3.6}
Let $W$ and $W'$ be basic submodules of $V$ with $W\cap W'=\{0\}.$ Then $W$ is (disjointly) orthogonal to $W'$ iff $b_{\alt}(W,W')=0.$
\end{lem}

\begin{proof}
This can be seen  exactly   as with the parallel \lemref{lem:2.5}.
Just replace in its proof   the quasiminimal companion of $q$  by $b_{\alt}.$
\end{proof}

\begin{defn}\label{defn:3.7} $ $
\begin{enumerate} \ealph
  \item
A \bfem{path} $\Gm$ in $V=(V,b)$ of \bfem{length} $r\ge1$ in $\mfB$ is a sequence $\veps_0,\veps_1,\dots,\veps_r$ of elements of $\mfB$ with
$$b_{\alt}(\veps_i,\veps_{i+1})\ne0\qquad (0\le i\le r-1).$$
In essence  this condition does not depend on the choice of the base $\mfB$, since $\mfB$ is unique up to multiplication  by units, and so we also say that \textbf{$\Gm$ is a path in $V$}.
We say that the path runs from $\veps:=\veps_0$ to $\eta:=\veps_r,$ or that the path connects $\veps$ to $\eta.$ A path of length $1$ is called an \bfem{edge}. This is just a pair $(\veps,\eta)$ in $\mfB$ with $\veps \neq \eta$ and $b(\veps,\eta)\ne0.$
\pSkip
 \item We define an equivalence   relation on $\mfB$ as follows. Given $\veps,\eta\in\mfB$, we declare that $\veps\sim\eta$ if either $\veps=\eta$ or there runs a path from $\veps$ to $\eta.$
\end{enumerate}

\end{defn}

It is now obvious how to mimic the theory of indecomposable
components from the end of \S\ref{sec:2} in the bilinear setting.

\begin{schol}\label{schol:3.8}
\thmref{thm:2.6} and its proof remain valid for the present equivalence relation on $\mfB.$
 We only have to replace the quasiminimal companion $b$ of $q$ there by $b_{\alt}$ and to use \lemref{lem:3.6}
 instead of \lemref{lem:2.5}. Again we denote the set of equivalence classes of $\mfB$ by $\{\mfB_k\ds|k\in K\}$ and the submodule of $V$ with base $\mfB_k$ by $V_k$, and again
 we call the $V_k$ the \bfem{indecomposable components} of $V.$ Also the
analog to Remark \ref{remarks:2.7} remains valid.
\end{schol}

We state a consequence of the parallel between the two decomposition
theories.

\begin{thm}\label{thm:3.9}
Assume that $(V,q)$ is a quadratic module with unique base and
$b$ is a quasiminimal companion of $q.$ The indecomposable
components of $(V,q)$ coincide with the indecomposable components of
$(V,b).$
\end{thm}

\begin{proof}
The equivalence relation used in \thmref{thm:2.6} is the same as the equivalence relation in Definition \ref{defn:3.7}. 
\end{proof}

We add an easy observation on bilinear modules.

\begin{prop}\label{prop:3.10} Assume that $(V,b)$ is a
bilinear $R$-module with unique base. A basic submodule $W$ of $V$
is indecomposable with respect to $b$, iff $W$ is indecomposable
with respect to $b_{\alt}.$
\end{prop}

\begin{proof} The equivalence relation on $\mfB$ just defined (Definition \ref{defn:3.7}) does not change if we replace $b$ by $b_{\alt}.$
\end{proof}
\section{Isometries, isotypical components, and a cancellation theorem}\label{sec:4}

Let $R$ be any semiring.

\begin{defn}\label{defn:4.1}  $ $
\begin{enumerate} \ealph
 \item For quadratic $R$-modules $V=(V,q)$ and $V'=(V',q')$,  an \bfem{isometry} $\sig: V\to V'$ is a
 bijective $R$-linear map with $q'(\sig x)=q(x)$ for all $x\in V.$ Likewise, if $V=(V,b)$ and $(V',b')$ are bilinear
 $R$-modules, an \bfem{isometry} is a bijective $R$-linear map $\sig: V\to V'$ with $b'(\sig x,\sig y)=b(x,y)$ for all $x,y\in V.$ \pSkip

 \item If there exists an isometry $\sig: V\to V'$, we call $V$ and $V'$ \bfem{isometric} and write $V\cong V'.$ We then also say that $V$ and $V'$ are \textit{in the same isometry class}.
     \end{enumerate}
 \end{defn}

In the following we study     quadratic and bilinear $R$-modules
with unique base on an equal footing.

It would  not hurt if we supposed that the semiring $R$ satisfies
the conditions in \thmref{thm:1.7}, so that every free $R$-module
has unique base, but the simplicity of all of the arguments in the
present section becomes more apparent if we do not rely on
\thmref{thm:1.7}.

\begin{notationdefinition}\label{notationdefinition:4.2} $ $
\begin{enumerate} \ealph
  \item
 Let $(V_\lm^0\ds|\lm\in \Lm)$ be a set of representatives of all isometry classes of indecomposable
 quadratic (resp. bilinear) $R$-modules with unique base \footnote{of rank bounded by the cardinality of $V$,
 in order to avoid set-theoretical complications}.
\pSkip

\item If $W$ is such an $R$-module, where $W\cong V_\lm^0$ for a unique $\lm\in\Lm$, we say that $W$ \bfem{has type} $\lm$ (or: $W$ is \bfem{indecomposable of type} $\lm).$
\pSkip

 \item We say that a quadratic (resp. bilinear) module $W\ne0$ with unique base is \bfem{isotypical} of type $\lm,$ if every indecomposable component  of $V$ has type $\lm.$
\pSkip

 \item Finally, given a quadratic (resp. bilinear) $R$-module with unique base, we denote the sum of all indecomposable components of $V$ of type $\lm$ by $V_\lm$ and call the $V_\lm\ne0$ the \bfem{isotypical components of} $V$.
\end{enumerate}
\end{notationdefinition}

The following is now obvious from \S\ref{sec:2} and \S\ref{sec:3}
(cf.~\thmref{thm:2.6} and Scholium \ref{schol:3.8}).

\begin{prop}\label{prop:4.3}
If $V$ is a quadratic or bilinear $R$-module with unique base, then
    $$V=\BigPerp_{\lm\in \Lm'} V_\lm $$
    with $\Lm'=\{\lm\in\Lm\ds|V_\lm\ne0\}.$
    \end{prop}

    Since our notion of orthogonality for basic submodules of $V$ is encoded in the linear and quadratic,
    resp.~bilinear, structure of $V,$ the following fact  also is obvious, but in view of its importance will be dubbed a ``theorem".

\begin{thm}\label{thm:4.4}
Assume that $V$ and $V'$ are quadratic (resp. bilinear) $R$-modules
with unique bases and $\sig: V\to V'$ is an isometry. Let
$\{V_k\ds|k\in K\}$ denote the set of indecomposable components of
$V.$
\begin{enumerate}\ealph
\item $\{\sig(V_k)\ds|k\in K\}$ is the set of indecomposable components of $V'.$ \pSkip
\item If $V_k$ has type $\lm$, then $\sig(V_k)$ has type $\lm,$ and so $\sig(V_\lm)=V_\lm'$ for every $\lm\in\Lm.$
   \end{enumerate}
      \end{thm}
    Also in the remainder of the section, we assume that the quadratic or bilinear modules have unique base.

    \begin{defn}\label{defn:4.5}
    Let $O(V)$ denote the group of all isometries $\sig:V\to V$ (i.e., automorphisms) of $(V,q)$, resp. $(V,b).$ As usual, we call $O(V)$ the \bfem{orthogonal group} of $V.$
    \end{defn}

    \thmref{thm:4.4} has the following immediate consequence.

    \begin{cor}\label{cor:4.6}
    Every $\sig\in O(V)$ permutes the indecomposable components of $V$ of fixed type~$\lm,$ and so $\sig(V_\lm)=V_\lm$ for every $\lm\in\Lm$.

    We have a natural isomorphism
 $$\xymatrix{\qquad \quad
O(V)\ar[r]^{1:1}  & \prod\limits_{\lm\in\Lm'}  O(V_\lm),}  $$
    sending $\sig\in O(V)$ to the family of its restrictions $\sig|V_\lm\in O(V_\lm).$
    \end{cor}

    \begin{defn}\label{defn:4.7} $ $ \begin{enumerate}\ealph
\item  Let $\lm\in\Lm$. We denote the cardinality of the set of indecomposable
components of $V_\lm$ by $m_\lm(V),$ and we call $m_\lm(V)$ the
\bfem{multiplicity} of $V_\lm.$ \{N.B. $m_\lm(V)$ can be infinite or
zero.\} \pSkip

 \item If $m_\lm\in\mathbb N_0$ for every $\lm\in\Lm,$ we say that $V$ is
 \bfem{isotypically finite}. \end{enumerate}\end{defn}

\begin{thm}\label{thm:4.8}
If $V$ and $V'$ are quadratic or bilinear $R$-modules with unique
bases, then $V\cong V'$ iff $m_\lm(V)=m_\lm(V')$ for every
$\lm\in\Lm.$\end{thm}

\begin{proof}
This follows from \propref{prop:4.3} and \thmref{thm:4.4}.
\end{proof}

We are ready for a main result of the paper.

\begin{thm}\label{thm:4.9}
Assume that
$W_1, W_2, W'_1, W'_2$  are quadratic or bilinear modules with unique base and that
$W_1$ is isotypically finite. Assume furthermore that
$W_1\cong W'_1$ and that $W_1\perp W_2 \cong W_1'\perp W_2'$. Then $W_2\cong W_2'.$

\end{thm}

\begin{proof}
For every $\lm\in\Lm$, clearly
 $m_\lm(V)=m_\lm(W_1)+m_\lm(W_2)$ and  $m_\lm(V')=m_\lm(W_1')+m_\lm(W_2')$.
 Since $V\cong V',$ the multiplicities $m_\lm(V)$ and $m_\lm(V')$ are equal, and since $W_1\cong W_1',$ the same holds for the multiplicities $m_\lm(W_1').$ Since $m_\lm(W_1)=m_\lm(W_1')$ is finite, it follows that $m_\lm(W_2)=m_\lm(W_2').$ By \thmref{thm:4.8} this implies that $W_2\cong W_2'.$
 \end{proof}

 \begin{remark}\label{remark:4.10}
If the free $R$-module $W_1$ has finite rank, then certainly $W_1$
is  isotypically finite. Thus Theorem \ref{thm:4.9} may be viewed as
the analogue of Witt's cancellation theorem from 1937 \cite{Witt}
proved for quadratic forms over fields.
\end{remark}

The assumption of isotypical finiteness in Theorem \ref{thm:4.9} cannot be relaxed. Indeed if $m_\lm(W_1)$ is infinite for at least one $\lm \in \Lm$, then the cancellation law becomes false.
This is evident by Theorem \ref{thm:4.8} and the following example.
 \begin{examp}\label{examp:4.11}
 Assume that $V$ is the orthogonal sum of infinitely many copies $V_1,V_2,\dots$ of an indecomposable quadratic or bilinear module $V_0$ with unique base. Consider the following submodules of~$V$:
 $$
 \begin{array}{lll}
   W_1:=&V_2\perp V_3\Perp\cdots,\qquad W_2:=V_1,\\[1mm]
  W_1':=&V_3\perp V_4\Perp \cdots,\qquad W_2':=V_1 \Perp V_2. \end{array}$$
  Then $W_1 \Perp W_2=V=W_1'\Perp W_2',$ and $W_1\cong W_1'.$ But $W_2$ is not isometric to $W_2'.$
  \end{examp}

\section{Expansions and tensor products}\label{sec:5}

Let $q:V\to R$ be a quadratic form on an $R$-module $V$. We recall from \cite[\S1]{QF1}
 that, when $V$ is free with base $(\veps_i \ds : i \in I)$, then $q$ admits a (not necessarily unique) \bfem{balanced companion}, i.e., a companion
  $b : V \times V \to R$ such that $b(x,x) = 2q(x)$ for all $x \in V$, and that it suffices to know for this that  $b(\veps_i,\veps_i) = 2q( \veps_i)$ for all $i \in I $ \cite[Proposition 1.7]{QF1}. Balanced companions are a crucial   ingredient
  in our definition below of a tensor product of a free bilinear module and a free quadratic module. They arise from ``expansions'' of $q$, defined as follows, cf.~ \cite[Definition 1.9]{QF1}.

  \begin{defn}\label{defn:5.1}
  A bilinear form $B:V\times V\to R$ (not necessarily symmetric) is an \bfem{expansion} of a balanced pair $(q,b)$ if $B+B^t=b,$ i.e.,
  \begin{equation}\label{eq:5.1}
  B(x,y)+B(y,x)=b(x,y)
  \end{equation}
  for all $x,y\in V,$ and
 \begin{equation}\label{eq:5.2}
  q(x) =B(x,x)
  \end{equation}
  for all $x\in V.$ If only the form $q$ is given and \eqref{eq:5.2} holds, we say that $B$ \textbf{is an expansion} of~$q.$
 \end{defn}

 As stated in the \cite[\S1]{QF1}, every bilinear form $B: V\times V\to R$ gives us a balanced pair $(q,b)$ via \eqref{eq:5.1} and \eqref{eq:5.2}, and, if the $R$-module $V$ is free, we obtain
all such pairs $(q,b)$   in this way. But   we will need  a
description of \textbf{all} expansions of $(q,b)$ in the free case.

 \begin{construction}\label{constr:5.2}
 Assume that $V$ is a free $R$-module and $(\veps_i \ds |i\in I)$ is a base of $V.$ When ~$(q,b)$ is a balanced pair on $V,$ we
  obtain all expansions $B:V\times V\to R$ of~$(q,b)$ as follows.

 Let $\alpha_i:=q(\veps_i),$ $\beta_{ij}:=b(\veps_i,\veps_j)$ for $i,j\in I.$ We have $\beta_{ij}=\beta_{ji}.$ We choose a total ordering on $I$ and for every $i< j$ two elements $\chi _{ij},\chi _{ji}\in R$ with
 $$\beta_{ij}=\chi _{ij}+\chi _{ji},\qquad (i< j).$$
 We furthermore put
 $$\chi _{ii}:=\alpha_i,$$
 and define $B$ by the rule
 $$B(\veps_i,\veps_j)=\chi _{ij}$$
 for all $(i,j)\in I\times I.$
 \end{construction}

 In practice one usually chooses $\chi _{ij}=\beta_{ij},$ $\chi _{ji}=0$ for $i<j,$ i.e., takes the unique ``triangular''  expansion $B$ of $(q,b),$ cf.~\cite[\S1]{QF1}, but now we do not want to depend on the choice of a total ordering of the base $(\veps_i \ds |i\in I).$ We used such an ordering above only to ease notation.
\pSkip

\textbf{Tensor products} over semirings in general require the use
of congruences \cite{H}, but for free modules the basics can be done
precisely as over rings, and  we  leave the formal details to the
interested reader.
 We only state here that, given two free $R$-modules $V_1$ and $V_2$, with  bases $\mfB_1$ and $\mfB_2$,  the
 $R$-module  $V_1 \otimes_R V_2 $ ``is'' the free $R$-module with base $\mfB_1 \otimes \mfB_2 $, which is a renaming of $\mfB_1 \times \mfB_2$, writing $\veps \otimes  \eta$ for $(\veps, \eta)$ with  $\veps \in \mfB_1$, $\eta \in \mfB_2.$ If
 $$ \mfB_1 = \{ \veps_i \ds | i \in I \}, \qquad  \mfB_2 = \{ \eta_j \ds | j \in J \}$$
  and $x= \sum\limits_{i \in I} \in V_1$ and  $y= \sum\limits_{j \in J} \in V_2$, we define, as common over rings,
  \begin{equation}\label{eq:5.x.1}
x \otimes y := \sum_{(i,j) \in I \times J} x_i y_j (\veps_i \otimes y_j),
 \end{equation}
  and this vector is independent  of the choice of the bases $\mfB_1$ and $\mfB_2$.
  If $B_1$ and $B_2$ are bilinear forms on $V_1$ and $V_2$ respectively, we have a well defined bilinear form on $V_1 \otimes_R V_2$, denoted by  $B_1 \otimes B_2$, such that for any $x_i \in V_1$, $y_j \in V_2$ ($i,j \in \{ 1,2\}$)
    \begin{equation}\label{eq:5.x.2}
(B_1 \otimes B_2)(x_1 \otimes x_2, y_1 \otimes y_2) = B_1(x_1,y_1) B_2(x_2, y_2).
 \end{equation}
 If $b_1$ and $b_2$ are symmetric bilinear forms on $V_1$ and $V_2$ respectively, then  $b_1 \otimes b_2$ is symmetric. Then we call the bilinear module $(V_1 \otimes_R V_2, b_1 \otimes b_2)$ the \textbf{tensor product of the bilinear modules} $(V_1, b_1)$ and $(V_2, b_2)$.

We next define the tensor product of a free bilinear and a free quadratic module.
The key fact which allows us to do this in a reasonable way is as follows.
 \begin{prop}\label{prop:5.3}
 Let $\gm :U\times U\to R$ be a symmetric bilinear form and $(q,b)$ a balanced quadratic pair on $V.$ Assume that $B$ and $B'$ are two expansions of $(q,b)$.
 Then the bilinear forms $\gm \otimes B$ and $\gm \otimes B'$ on $U\otimes V$ yield the same balanced pair $(\tlq,\tlb)$ on $U\otimes V.$ We have $\tlb=\gm \otimes b,$ whence for $u_1,u_2\in U,$ $v_1,v_2\in V,$
 \begin{equation}\label{eq:5.3}
 \tlb(u_1\otimes v_1,u_2\otimes v_2)=\gm (u_1,u_2)b(v_1,v_2).
 \end{equation}
 Furthermore, for $u\in U$ and $v\in V,$
 \begin{equation}\label{eq:5.4}
 \tlq(u\otimes v)=\gm (u,u)q(v).
 \end{equation}
  \end{prop}

  \begin{proof}
  $\gm \otimes B+(\gm \otimes B)^t=\gm \otimes B+\gm ^t\otimes B^t=\gm \otimes B+\gm \otimes B^t=\gm \otimes(B+B^t)=\gm \otimes b.$ Also $\gm \otimes B'+(\gm \otimes B')^t=\gm \otimes b.$ Furthermore,
  $$(\gm \otimes B)(u\otimes v,u\otimes v)=\gm (u,u)B(v,v)=
  \gm (u,u)q(v)=(\gm \otimes B')(u\otimes v,u\otimes v)$$
  for any $u\in U,$ $v\in V.$ Together these equations  imply $(\gm \otimes B)(z,z)=(\gm \otimes B')(z,z)$
for any $z\in U\otimes V.$
  \end{proof}

   \begin{defn}\label{defn:5.4}
   We call $\tlq$ the \bfem{tensor product} of the bilinear form $\gm $ and the quadratic form $q$ \bfem{with respect to the balanced companion} $b$ of $q,$ and write
   $$\tlq=\gm \otimes_b q,$$
   and we also write $\tlV=U\otimes_b V $ for the quadratic $R$-module $\tlV=(U\otimes V,\tlq).$
   \end{defn}

   \begin{remark}\label{rem:5.5}
   If $q$ has only one balanced companion, we may suppress the ``$b$'' here, writing $\tlq=\gm \otimes q.$ Cases in which this happens are:
   $q$ is rigid, $V$ has rank  one,  $R$ is embeddable in a ring. 
   \end{remark}

   \begin{prop}
  \label{prop:5.6}
  If $U=(U,\gm )$ has an orthogonal decomposition $U= \bigPerp\limits_{i\in I} U_i,$ then
  $$U\otimes_b V = \BigPerp_{i\in I}U_i\otimes_b V.$$
  \end{prop}

\begin{proof}
It is immediate that
 $(\gm \otimes b)(U_i\otimes V,U_j\otimes V)=0$ for $i\ne j.$
 \end{proof}

We proceed to explicit examples. For this we need notation from \cite[\S1]{QF1} which we recall for the convenience of the reader.

Assume that $V$ is free of finite rank $n$ and $\mfB$ is a base of $V$ for which we now choose a total ordering, $\mfB = (\veps_1, \veps_2, \dots, \veps_n)$. Then we identify a bilinear form $B$ on $V$ with the $(n\times n)$-matrix
\begin{equation}\label{eq:5.y.1}
B=\begin{pmatrix} \bt_{11} & \bt_{12}& \cdots & \bt_{1n}\\
\bt_{21}& \bt_{22} & & \bt_{2n}\\
 \vdots  & \vdots  &\ddots & \vdots \\
\bt_{n1}& \cdots & & \bt_{nn}\end{pmatrix}, \end{equation}
where $\bt_{ij} = B(\veps_1, \veps_j).$ In particular, a bilinear $R$-module $(V,\bt)$ is denoted by a symmetric $(n \times n )$-matrix, namely its Gram matrix $b = (\bt_{ij})_{1 \leq i,j \leq n}$, where $\bt_{ij} = \bt_{ji} = b(\veps_i, \veps_j)$.

Given a quadratic module $(V,q)$, we choose a triangular expansion
\begin{equation}\label{eq:5.y.2}
B=\begin{pmatrix} \al_{1} & \al_{12}& \cdots & \al_{1n}\\
0 & \al_{2} & \cdots & \al_{2n}\\
 \vdots  & &\ddots & \vdots \\
0 & \cdots & 0 & \al_{n}\end{pmatrix} \end{equation} of $q$ and
denote $q$ by the triangular scheme
\begin{equation}\label{eq:5.y.3}
q=\begin{bmatrix} \al_{1} & \al_{12}& \cdots & \al_{1n}\\
& \al_{2} & \cdots & \al_{2n}\\
   & &\ddots & \vdots \\
 &  &  & \al_{n}\end{bmatrix} \end{equation}
so that $q$ is given by the polynomial
$$ q(x) = \sum_{i=1}^n \al_i x_i^2 + \sum_{i< j}^n \al_{ij} x_i x_j.$$
(Such triangular schemes have already been used in the literature
when $R$ is a ring, e.g.~\cite[~I~\S2]{Kneser}.) In the case that $q$
is diagonal, i.e., all $\al_{ij}$ with $i< j$ are zero, we usually
write instead of~\eqref{eq:5.y.2} the single row
\begin{equation}\label{eq:5.y.4} q=[\al_1,\al_2,\dots, \al_n].\end{equation}
Analogously we use for a diagonal symmetric bilinear form $b$ (i.e., $b(\veps_i, \veps_j) = 0$ for $i \neq j$) the notation
\begin{equation}\label{eq:5.y.5} b= \langle\bt_{11},\bt_{22},\dots, \bt_{nn}\rangle.\end{equation}

We note that the quadratic form \eqref{eq:5.y.3} has the balanced companion \begin{equation}\label{eq:5.y.6}
b=\begin{pmatrix} \al_{1} & \al_{12}& \cdots & \al_{1n}\\
\al_{12}& \al_{2} & & \al_{2n}\\
 \vdots  & \vdots  &\ddots & \vdots \\
\al_{1n}& \cdots & & \al_{n}\end{pmatrix} \end{equation} and
\eqref{eq:5.y.4}, being diagonal, has the unique (!) balanced
companion
\begin{equation}\label{eq:5.y.7} b= \langle 2 \al_{1}, 2 \al_{2},\dots, 2 \al_{n}\rangle.\end{equation}

 \begin{example}\label{examp:5.7}
 If $a_1,\dots, a_n,$ $c\in R,$ then
 \begin{equation}\label{eq:5.5}
 \langle a_1,\dots a_n\rangle\otimes[c]=[a_1c,\dots,a_nc].
 \end{equation}
 This is evident from \propref{prop:5.6} and the rule $\langle a\rangle\otimes[c]=[ac]$ for one-dimensional forms which holds by \eqref{eq:5.4}. In particular
  \begin{equation}\label{eq:5.6}
[a_1,\dots,a_n]= \langle a_1,\dots a_n\rangle\otimes[1].
 \end{equation}
 \end{example}

 \begin{example}\label{examp:5.8} (As before, $R$ is any semiring.)
 Assume that $V=(V,q)$ has dimension $n$, and take a base $\eta_1,\dots, \eta_n$ of $V.$ Let $$(U,\gm )=\begin{pmatrix} 0 & 1\\ 1&0\end{pmatrix}$$ with base $\veps_1,\veps_2.$ We choose a balanced companion $b$ of $V,$ written as a symmetric $(n\times n)$-matrix $(b(\eta_i,\eta_j)).$ We see by the use of the rules \eqref{eq:5.3} and \eqref{eq:5.4} that
 \begin{equation}\label{eq:5.7}
 \begin{pmatrix} 0 & 1\\ 1&0\end{pmatrix}\otimes_bq=\begin{bmatrix}[c|c]
  0 & b\\
  \hline
   & 0
\end{bmatrix}
 \end{equation}
 written with respect to the base
 $$\veps_1\otimes \eta_1,\dots,\veps_1\otimes \eta_n,\veps_2\otimes \eta_1,\dots, \veps_2\otimes\eta_n.$$
 \end{example}

 This example illustrates dramatically that in general \textbf{the tensor product of $\gm $ and $q$ depends on the chosen balanced companion $b$ of $q$}: tensoring $q$ by $ \left(\begin{smallmatrix} 0 & 1\\ 1&0\end{smallmatrix} \right)$ produces the symmetric matrix of $b.$

 \begin{remark}\label{remark:5.9}
 If $\gm _1$ and $\gm _2$ are bilinear forms on the same free $R$-module $U$, then the rules~ \eqref{eq:5.3} and \eqref{eq:5.4} imply for any $\lm_1,\lm_2\in R$ that
 \begin{equation}\label{eq:5.8}
 (\lm_1\gm _1+\lm_2\gm _2)\otimes _bq=\lm_1(\gm _1\otimes_bq)+\lm_2(\gm _2\otimes_bq).
 \end{equation}
 \end{remark}

 \begin{example}\label{examp:5.10}
 Using \eqref{eq:5.8} with
 $$\gm _1=\langle a_1,a_2\rangle,\quad \gm _2=\begin{pmatrix} 0 & 1\\ 1 & 0\end{pmatrix},\quad \lm_1=1,\quad \lm_2=\lm,$$
 we obtain from \propref{prop:5.6} and Example \ref{examp:5.7} that
 \begin{equation}\label{eq:5.9}
 \begin{pmatrix} a_1 & \lm\\ \lm & a_2\end{pmatrix}\otimes_b q=
 \begin{bmatrix}[c|c]
  a_1q & \lm b\\
  \hline
   & a_2q
\end{bmatrix}
\end{equation}
 \end{example}

\begin{example}\label{examp:5.11}
Let
$$q=\begin{bmatrix} 0 & a_{12} &\cdots & a_{1n}\\
& \ddots&\ddots &  \vdots\\
& & &   a_{n-1,n}\\
& &  &  0
\end{bmatrix}$$
with $a_{ij}\in R$ $(i<j).$ Then $q$ is rigid (cf.~\cite[Proposition
3.4]{QF1}; no assumption on $R$ is needed here).  Furthermore, let
$$\gm =\begin{pmatrix} \gm _{11} & \cdots&\gm _{1m}\\
\vdots &  & \vdots\\
\gm _{m1}&\cdots & \gm _{mm}\end{pmatrix}$$
with $\gm _{ij}=\gm _{ji}\in R.$ Then we obtain by the rules \eqref{eq:5.3} and  \eqref{eq:5.4} that
\begin{equation}\label{eq:5.10}
\gm \otimes q \ds = \
 \begin{matrix}[|c|c|c|c|]
    \hline
   0 & a_{12}\gm  & \cdots & a_{1n}\gm \\
     \hline
    & 0 & & a_{2n}\gm \\
   \hline
   & & \ddots & a_{n-1,n}\gm \\
   \hline
   & & & 0\\
   \hline
\end{matrix}
\end{equation}
More precisely, if the presentations of $q$ and $\gm $ above refer to ordered bases $(\eta_1,\dots,\eta_n)$ and $(\veps_1,\dots,\veps_m),$ respectively, then  \eqref{eq:5.10} refers to the ordered base
$$(\veps_1\otimes \eta_1,\dots,\veps_m\otimes\eta_1,\veps_1\otimes\eta_2,\dots,\veps_m\otimes \eta_n).$$
\end{example}

We now consider the tensor product $\gm \otimes[a]=\gm
\otimes_b[a],$  cf.~Equation~\eqref{eq:5.y.4}, where $b$ is the
unique balanced companion of $[a]$, \eqref{eq:5.y.7}. Our starting
point is a definition which makes sense for any semiring $R$ and any
$R$-module $U.$

\begin{defn}\label{defn:5.12}
Let $\gm :U\times U\to R$ be a symmetric bilinear form. The \bfem{norm form} of $\gm $ is the quadratic form $n(\gm ):U\to R$ with
$$n(\gm )(x):=\gm (x,x)$$
for any $x\in U.$
\end{defn}

\begin{remark}\label{remark:5.13}
The norm form $n(\gm )$ has the expansion $\gm : U\times U\to R$ and
the associated balanced companion $\gm +\gm ^\t=2\gm .$ The norm
forms are precisely all the quadratic forms which admit a symmetric
expansion. If $U$ has a finite base $\veps_1,\dots,\veps_n$, then
with respect to this base
\begin{equation}\label{eq:5.11}
n(\gm )=\begin{bmatrix} \gm _{11} & 2\gm _{12} &\cdots & 2\gm _{1m}\\
& \gm _{22} & &\\
& &\ddots&\vdots\\
& & &\gm _{mm}
\end{bmatrix},
\end{equation}
where $\gm _{ij}:=\gm (\veps_i,\veps_j).$
\end{remark}

\begin{prop}\label{prop:5.14}
Assume that $U=(U,\gm )$ is a free bilinear $R$-module and $a\in R.$ Then
\begin{equation}\label{eq:5.12}
U\otimes [a]\cong(U,\ a\, n(\gm )).
\end{equation}
\end{prop}

\begin{proof}
We realize the form $[a]$ as a quadratic module $(V,q)$ with
$V=R\eta$ free of rank 1 and $q(\eta)=a.$ $\{q$ has the unique
balanced companion $b:V\times V\to R$, with $b(\eta,\eta)=2a.\}$ The
form $\tlq:=\gm \otimes q=\gm \otimes_bq$ is given by
$$\tlq(x\otimes\eta)=\gm (x,x)a=(an(\gm ))(x).$$
The claim is obvious.
\end{proof}

\begin{examp}\label{examp:5.15}
Assume that $U$ has base $\veps_1,\dots,\veps_m.$ Let $\gm
_{ij}:=\gm (\veps_i,\veps_j).$ Then $$\gm \otimes[a]\cong(a\gm
)\otimes [1],$$ and
\begin{equation}\label{eq:5.13}
\gm \otimes [1]=\begin{bmatrix}\gm _{11} & 2\gm _{12} &\cdots & 2\gm _{1n}\\
&\gm _{22} & &\\
& &\ddots &\vdots\\
& & & \gm _{mm}
\end{bmatrix},
\end{equation}
where the right hand side refers to the base $\veps_1\otimes\eta,\,
\veps_2\otimes\eta,\,\dots,\, \veps_m\otimes\eta.$
\end{examp}

At a crucial point in \S\ref{sec:6}, we will need an explicit
description of the tensor products $\gm \otimes_bq$ with~ $q$
indecomposable of rank 2. We start with a general fact.

\begin{prop}\label{prop:5.16}
Assume that $\gm $ is a symmetric bilinear form on a free $R$-module $U$ and $q_1,q_2$ are quadratic forms on a free $R$-module $V.$ Let $b_1,b_2$ be
 balanced companions of $q_1$ and~$q_2,$ respectively. Let $q:=\lm_1q_1+\lm_2q_2$ with $\lm_1,\lm_2\in R.$ Then $b:=\lm_1b_1+\lm_2b_2$ is a balanced companion of $q,$ and
\begin{equation}\label{eq:5.14}
\gm \otimes_bq=\lm_1(\gm \otimes_{b_1}q_1)+\lm_2(\gm \otimes_{b_2}q_2).
\end{equation}
This form has the balanced companion $\gm \otimes b$ (as we know) and
\begin{equation}\label{eq:5.15}
\gm \otimes b=\lm_1(\gm \otimes b_1)+\lm_2(\gm \otimes b_2).
\end{equation}
\end{prop}

\begin{proof}
An easy check by use of \eqref{eq:5.3} and \eqref{eq:5.4}.
\end{proof}

\begin{examp}\label{examp:5.17}
We take a free module $V$ with base $\eta_1,\eta_2$, and choose with
 respect to this base
$$q_1=\begin{bmatrix} a_1 &0\\  & a_2\end{bmatrix}=[a_1,a_2],\qquad q_2=\begin{bmatrix} 0 & c\\ & 0
\end{bmatrix}$$
with $a_1,a_2,c\in R,$ $c\ne0,$ and the balanced companions
$$b_1=\begin{pmatrix} 2a_1 & 0\\ 0 & 2a_2\end{pmatrix},\qquad b_2=\begin{pmatrix} 0 & c\\ c& 0\end{pmatrix}.$$
Then
$$q:=q_1+q_2=\begin{bmatrix} a_1 & c\\ & a_2\end{bmatrix}
$$
has the balanced companion
$$b:=b_1+b_2=\begin{pmatrix} 2a_1 & c\\ c& 2a_2\end{pmatrix}.$$
For $$\gm =\begin{pmatrix} \gm _{11} &\cdots &\gm _{1m}\\  \vdots &&
\vdots \\ \gm _{m1} &\cdots & \gm _{mm}\end{pmatrix}$$  on a free
module $U$ with to the base $\veps_1,\dots,\veps_m$, we get
$$\gm \otimes_{b_1}q_1=\begin{bmatrix}[c|c] a_1n(\gm ) & 0\\ \hline &a_2n(\gm )\end{bmatrix},\qquad
\gm \otimes_{b_2}\begin{bmatrix} 0 &c\\ &
0\end{bmatrix}=\begin{bmatrix}[c|c] 0 & c\gm \\\hline &
0\end{bmatrix}, \qquad \text{cf.~\eqref{eq:5.10}},$$
 and finally
 \begin{equation}\label{eq:5.16}
 \gm \otimes_b\begin{bmatrix}a_1 &c\\ & a_2\end{bmatrix}= \begin{bmatrix}[c|c]
  a_1n(\gm ) & c\gm \\
  \hline
   & a_2n(\gm )
\end{bmatrix}
\end{equation}
with respect to the base
$$\veps_1\otimes\eta_1,\dots,\veps_m\otimes\eta_1,\veps_1\otimes\eta_2,\dots,\veps_m\otimes
\eta_2.$$
\end{examp}

\begin{remark}\label{rem:5.18}
From \eqref{eq:5.16} and \eqref{eq:5.9}, we obtain the useful
formula
\begin{equation}\label{eq:5.17}
\gm \otimes_b\begin{bmatrix} a_1 & c\\ & a_2\end{bmatrix}=\begin{pmatrix} a_1 & c\\ c& a_2\end{pmatrix} \otimes _{2\gm } n(\gm )\ ,
\end{equation}
by use of Example \ref{examp:5.10} for the quadratic pair $(n(\gm ),2\gm ).$
\end{remark}

\textbf{From now on, we assume that $V$ has unique base.} \{We do not need that $U$ has unique base.\}

\begin{defn}\label{defn:5.19} We call a companion $b$ of $q$ \bfem{faithful} if $b$ is balanced and quasiminimal.
\end{defn}

\begin{prop}\label{prop:5.20}
Assume that $b$ is a faithful companion of $q,$ and that $V=W_1\perp
W_2$ is an orthogonal decomposition of $V.$ Then, writing
$U\otimes_b W_i$ instead of $U\otimes_{(b|W_i)} W_i$, we have
$$U\otimes_b V=U\otimes_bW_1\perp U\otimes_bW_2$$
for any bilinear $R$-module $U.$
\end{prop}

\begin{proof}
$b(W_1,W_2)=0,$ since $b$ is quasiminimal. It follows that
$$(\gm \otimes b)(U\otimes W_1,U\otimes W_2)=0.$$
Thus, $\tlq=\gm \otimes_bq$ is quasilinear on $(U\otimes W_1)\times(U\otimes W_2).$
\end{proof}

\begin{examp}\label{examp:5.21}
Our assumption, that $b$ is faithful, is necessary here. If $V=W_1\perp W_2,$ and $b$ is balanced, but
$b(W_1,W_2)\ne0$, then $$\begin{pmatrix} 0 &1\\ 1 &0\end{pmatrix}\otimes_bV=\begin{bmatrix} 0 & b\\ &0
\end{bmatrix} $$ is not the orthogonal sum of $$ \begin{pmatrix} 0 & 1\\ 1 & 0\end{pmatrix}\otimes_b W_1 \dss{ \text{ and } }
 \begin{pmatrix} 0 & 1\\ 1 & 0\end{pmatrix}\otimes_b W_2.$$
 \end{examp}

 \begin{examp}\label{examp:5.22}
Let $q=[a_1,a_2,\dots,a_n]$ be a diagonal quadratic form. The diagonal symmetric bilinear form
$$b:=\langle 2q_1,\dots,2a_n\rangle$$
is the unique faithful companion of $q.$
For any bilinear $R$-module $(U,\gm )$, we have
\begin{equation}\label{eq:5.18}
\gm \otimes_bq=\gm \otimes[a_1]\perp\cdots \perp\gm \otimes[a_n].
\end{equation}
 \end{examp}
Concerning the forms $\gm \otimes[a_i],$ recall \propref{prop:5.14} and Example \ref{examp:5.15}.

\section{Indecomposability in tensor products}\label{sec:6}

In this section, we assume for simplicity  that $R\setminus\{0\}$
\textbf{is an entire antiring}. So every free $R$-module has unique
base (cf.~\thmref{thm:1.3}), and $R$ has no zero divisors. We
discuss decomposability first in tensor products of (free) bilinear
modules, later in tensor products of bilinear modules with quadratic
modules.

Let $V_1=(V_1,b_1)$ and $V_2=(V_2,b_2)$ be indecomposable free
(symmetric) bilinear modules over $R,$ and let $V:=V_1\otimes
V_2=(V_1\otimes V_2,b)$ with $b:=b_1\otimes b_2.$ We take bases
$\mfB_1$ and $\mfB_2$ of the $R$-modules $V_1,V_2$ respectively and
then have the base $$\mfB=\mfB_1\otimes\mfB_2:=\{\veps\otimes\eta
\ds|\veps\in \mfB_1,\eta\in\mfB_2\}$$ of $V.$ Our task  is to
determine the indecomposable components of $V.$ First we discuss the
``trivial'' cases.

\begin{remark}\label{rem:6.1}
Assume that $V_1$ has dimension ($=$ rank) one, so $V_1\cong\langle a\rangle$ with $a\in R.$ If $a\ne0,$ then $V$ is clearly indecomposable. If $a=0,$ then $b_1\otimes b_2=0,$ whence $V$ is indecomposable only if also $\dim V_2=1.$ Then $V=\langle 0\rangle.$
\end{remark}

In all the following, we assume that $V_1\ne\langle0\rangle,$ $V_2\ne\langle0\rangle.$

We resort to \S\ref{sec:3} to describe bases of the indecomposable
components of $V=(V,b)$ as the classes in $$\mfB=\{\veps\otimes\eta
\ds |\veps\in\mfB_1,\eta\in\mfB_2\}$$ of an equivalence relation
given by ``paths'', cf.~Definition \ref{defn:3.7}. So a path of
length $r\ge1$ in~$V,$ i.e., in $\mfB,$ is a sequence
\begin{equation}\label{eq:6.1}
\Gm=(\veps_0\otimes\eta_0,\veps_1\otimes\eta_1,\dots,\veps_r\otimes\eta_r)
\end{equation}
with
\begin{equation}\label{eq:6.2}
b_1(\veps_i,\veps_{i+1})b_2(\eta_i,\eta_{i+1})\ne0
\end{equation}
and
\begin{equation}\label{eq:6.3}
\veps_i\ne\veps_{i+1}\qquad\text{or}\qquad \eta_i\ne\eta_{i+1}
\end{equation}
for $0\le i\le r-1.$

Let us \textit{first assume that both} $b_1$ \textit{and} $b_2$ \textit{are alternate}, whence  also $b=b_1\otimes b_2$ is alternate. Now condition
\eqref{eq:6.3} is a consequence of \eqref{eq:6.2} and thus can be ignored. We read off from~ \eqref{eq:6.2} that
\begin{equation}\label{eq:6.4}
\Gm_1=(\veps_0,\veps_1,\dots,\veps_r), \quad \Gm_2=(\eta_0,\eta_1,\dots,\eta_r)
\end{equation}
are paths in $V_1$ and $V_2$ respectively of same length $r.$ Conversely, given such paths $\Gm_1$ and $\Gm_2$, they combine to a path $\Gm$ of length $r$ in $V,$ as written in \eqref{eq:6.1}. \{Here we use the assumption that $R$ has no zero divisors.\} We write
\begin{equation}\label{eq:6.5}
\Gm=\Gm_1\otimes\Gm_2.
\end{equation}

We will speak of ``cycles'' in $\mfB_1,$ $\mfB_2,$ $\mfB,$ in the following obvious way:

\begin{defn}\label{defn:6.2} Let $\mfC$ be a base of a free bilinear $R$-module $W.$
\begin{enumerate} \ealph \dispace
  \item
 We denote the length of a path $\Gm$ in $\mfC$ by $\ell(\Gm).$

\item A \bfem{cycle} $\Dl$ in $W$ with base point $\zeta\in\mfC$ is a path $(\zeta_0,\zeta_1,\dots,\zeta_r)$ in $\mfC$ with $\zeta_0=\zeta_r=\zeta.$ We say that the cycle $\Dl$ is \bfem{even} (resp. \bfem{odd}) if $\ell(\Dl)$ is \bfem{even} (resp. \bfem{odd}). We say that $\Dl$ is a \bfem{$2$-cycle} if $\ell(\Dl)=2,$ whence $\Dl=(\zeta,\zeta',\zeta)$ with $(\zeta,\zeta')$ an edge.
\end{enumerate}
\end{defn}

\begin{lemma}\label{lem:6.3}
Let $\veps,\veps'\in\mfB_1$ and $\eta,\eta'\in\mfB_2.$
Let $\Gm_1$ be a path from $\veps$ to $\veps'$ of length $r$ and ~$\Gm_2$ a path from $\eta$ to $\eta'$ of length $s,$ and assume that $r\equiv s\pmod2.$ Then $\veps\otimes \eta\sim \veps'\otimes\eta'.$
\end{lemma}

\begin{proof}
Assume, without loss of generality, that $s\ge r$, whence $s=r+2t$ with $t\ge0.$ If $t=0,$ then $\Gm_1\otimes\Gm_2$ is a path from $\veps\otimes\eta$ to $\veps'\otimes \eta'$ in $V.$ If $t>0$, we replace $\Gm_1=(\veps_0,\veps_1,\dots,\veps_r)$ by
$$\tlGm_1=(\veps_0,\veps_1,\dots,\veps_r,\veps_{r-1},\veps_r,\dots)$$
adjoining $t$ copies of the 2-cycle $(\veps_r,\veps_{r-1},\veps_r)$ to $\Gm_1.$ Now $\tlGm_1\otimes \Gm_2$ runs from $\veps\otimes\eta$ to $\veps'\otimes\eta'.$
\end{proof}

\begin{thm}\label{thm:6.4}
Assume that both $b_1$ and $b_2$ are alternate (and $V_1\ne\langle0\rangle,$ $V_2\ne\langle0\rangle,$ as always).
\begin{enumerate} \dispace
\item[a)] If $V_1$ or $V_2$ contains an odd cycle, then $V_1\otimes V_2$ is indecomposable.
\item[b)] Otherwise $V_1\otimes V_2$ is the orthogonal sum of two indecomposable components.
\end{enumerate}
\end{thm}

\begin{proof}
a): We assume that $V_1$ contains an odd
cycle $\Dl$ with base point $\dl.$ Let $\veps\otimes\eta$ and $\veps'\otimes\eta'$ be different elements of $\mfB.$ We want to verify that $\veps\otimes\eta\sim\veps'\otimes\eta'.$ We choose a path $\Gm_1$ from $\veps$ to $\veps'$ in $V_1$ and a path $\Gm_2$ from $\eta$ to $\eta'$ in $V_2.$ If $\ell(\Gm_1)\equiv\ell(\Gm_2)  \pmod 2$, then we know by \lemref{lem:6.3} that $\veps\otimes\eta\sim\veps'\otimes\eta'.$ Now assume that $\ell(\Gm_1)$ and $\ell(\Gm_2)$ have different parity. We choose a new path $\tlGm_1$ from $\veps$ to $\veps'$ as follows: We first take a path $H$ from $\veps$ to the base point $\dl$ of $\Dl,$ then we run through $\Dl,$ then we take the path inverse to $H$ (in the obvious sense) from $\dl$ to $\veps,$ and finally we run through $\Gm_1.$ The length $\ell(\tlGm_1)$ has different parity than $\ell(\Gm_1)$ and thus the same parity as $\ell(\Gm_2).$ We conclude again that $\veps\otimes\eta\sim\veps'\otimes \eta'.$
\pSkip

b): Now assume that both $V_1$ and $V_2$ contain only even cycles. This means that both in $V_1$ and $V_2$ all paths from a fixed start to a fixed end have length of the same parity. Given $\veps\otimes\eta$ and $\veps'\otimes\eta'$ in $\mfB$, every path $\Gm$ from $\veps\otimes\eta$ to $\veps'\otimes\eta'$ has the shape $\Gm_1\otimes\Gm_2$ with $\Gm_1$ running from $\veps $ to $\veps',$ $\Gm_2$ running from $\eta$ to $\eta'$, and $\ell(\Gm_1)=\ell(\Gm_2).$ Thus, if the paths from $\veps$ to $\veps'$ have length of different parity than those from $\eta$ to $\eta',$ then $\veps\otimes\eta$ cannot be connected to $\veps'\otimes\eta'$ by a path. But $\veps\otimes\eta$ can be connected to $\veps'\otimes\eta'',$ where $\eta''$ arises from $\eta'$ by adjoining an edge at the endpoint of $\eta'$. We fix some $\veps_0\in\mfB_1,$ and $\eta_0,\eta_1\in\mfB _2$ with $b_2(\eta_0,\eta_1)=1.$ Then every element of $\mfB$ can be connected by a path to $\veps_0\otimes\eta_0$ or to $\veps_0\otimes\eta_1,$ but not to both. $V$ has exactly two indecomposable components.
\end{proof}

\begin{remark}\label{rem:6.5}
Assume again that $b_1$ and $b_2$ are alternate and $\mfB_1$ and $\mfB_2$ both contain only even cycles. Let $\veps,\veps'\in\mfB_1$ and $\eta,\eta'\in\mfB_2$, and choose paths $\Gm_1$ from $\veps$ to ~$\veps'$ and $\Gm_2$ from $\eta$ to~ $\eta'.$ As the proof of \thmref{thm:6.4}.b has shown, $\veps\otimes\eta$ and $\veps'\otimes\eta'$ lie in the same indecomposable component of $V_1\otimes V_2$ iff $\ell(\Gm_1)$ and $\ell(\Gm_2)$ have the same parity.
\end{remark}

There remains the case that $b_1$ or $b_2$ is not alternate.

\begin{thm}\label{thm:6.6}
Assume that $b_1$ is \bfem{not} alternate and -- as before -- that $V_1=(V_1,b_1)$ and $V_2=(V_2,b_2)$ are indecomposable. Then $(V_1\otimes V_2,b_1\otimes b_2)$ is indecomposable.
\end{thm}

\begin{proof}
Every path in $V:=V_1\otimes V_2$ with respect to $(b_1)_{\alt}\otimes(b_2)_{\alt}$ is also a path with respect to $b_1\otimes b_2,$ as is easily checked, and the paths in $V_i$ with respect to $b_i$ are the same as those with respect to $(b_i)_{\alt}$ $(i=1,2).$ Thus we are done by \thmref{thm:6.4}, except in the case that all cycles in $V_1$ and in $V_2$ are even. Then $V$ has two indecomposable components $W',$ $W''$ with respect to $(b_1)_{\alt}\otimes(b_2)_{\alt}.$ The base
$$\mfB=\mfB_1\otimes\mfB_2:=(\veps\otimes\eta \ds|\veps\in\mfB_1,\eta\in\mfB_2)$$
of $V_1\otimes V_2$ is the disjoint union of sets $\mfB',$ $\mfB''$ which are bases of $W'$ and $W''.$ Any two elements of $\mfB'$ are connected by a path with respect to $(b_1)_{\alt}\otimes (b_2)_{\alt},$ hence by a path with respect to $b_1\otimes b_2,$ and the same holds for the set $\mfB''.$

We choose some $\rho\in\mfB_1$ with $b_1(\rho,\rho)\ne0$ and an edge
$(\eta_0,\eta_1)$ in $\mfB''.$ Since $R$ has no zero divisors, it
follows that $(\rho\otimes\eta_0,\rho\otimes\eta_1)$ is an edge in
$\mfB$ with respect to $b_1\otimes b_2.$ Perhaps interchanging $W'$
and $W''$, we assume that $\rho\otimes \eta_0\in\mfB'.$ Suppose that
also $\rho\otimes\eta_1\in\mfB'.$ Then there exists a path $\Gm$ in
$\mfB'$ with respect to $(b_1)_{\alt}\otimes(b_2)_{\alt}$ running
from $\rho\otimes\eta_0$ to $\rho\otimes\eta_1.$ $\Gm$~ has the form
$\Gm_1\otimes\Gm_2$, with $\Gm_1$ a cycle in $V_1$ with base point
$\rho$, and $\Gm_2$ a path in $V_2$ running from $\eta_0$ to
$\eta_1.$ We have $\ell(\Gm_1)=\ell(\Gm_2)$ and $\ell(\Gm_2)$ is
even. But there exists the path $(\eta_0,\eta_1)$ from $\eta_0$ to
$\eta_1$ of length 1. Since all paths in $V_2$ from $\eta_0$ to
$\eta_1$ have the same parity, we infer that $\ell(\Gm_2)$ is odd, a
contradiction.

We conclude that $\rho\otimes\eta_1\in\mfB''.$ The elements $\rho\otimes\eta_0\in\mfB'$ and $\rho\otimes \eta_1\in\mfB''$ are connected by a path with respect to $b_1\otimes b_2,$ and thus all elements of $\mfB$ are connected by paths with respect to $b_1\otimes b_2.$
\end{proof}

Turning to a study of indecomposable components of tensor products of bilinear and quadratic modules, we need some more terminology. Let $V=(V,q)$ be a free quadratic $R$-module and $\mfB$ a base of $V.$ We focus on balanced companions of $q.$

\begin{defn}\label{defn:6.7}
$ $
\begin{enumerate} \ealph \dispace
  \item  We call a companion $b$ of $q$ \bfem{faithful} if $b$ is \bfem{balanced and quasiminimal} (cf.~\S\ref{sec:2} above), whence $b(\veps,\veps)=2q(\veps)$ for all $\veps\in\mfB$ and $b(\veps,\eta)=0$ for
   $\veps\ne \eta$ in $\mfB$ such that $q$ is quasilinear on $R\veps\times R\eta.$

 \item  Given a balanced companion $b$ of $q$, we define a new bilinear form $b_f$ on $V$ by the rule that, for $\veps,\eta\in\mfB,$
$$b_f(\veps,\eta)=\begin{cases}0&\quad\text{if $\veps\ne\eta$ and $q$ is quasilinear on $R\veps\times R\eta$},\\
b(\veps,\eta)&\quad\text{else}.\end{cases}$$

\end{enumerate}
\end{defn}

It is clear from \cite[Theorem 6.3]{QF1} that again $b_f$ is  a
companion of $q.$ By definition, this companion is quasiminimal.
$b_f$ is also balanced, since
$b_f(\veps,\veps)=b(\veps,\veps)=2q(\veps)$ for all $\veps\in\mfB,$
cf.~\cite[Proposition 1.7]{QF1}, and so $b_f$ is faithful. We call
$b_f$ the \bfem{faithful companion of}~$q$ \bfem{associated to} $b.$

\begin{thm}\label{thm:6.8}
Assume that $b$ is a balanced companion of $q,$ and that $W$ is a basic submodule of $V.$ Then $W$ is indecomposable with respect to $q$ iff $W$ is indecomposable with respect to $b_f.$
\end{thm}

\begin{proof} This is a special case of \thmref{thm:3.9}, since $b_f|W=(b|W)_f$ is a quasiminimal companion of $q|W.$
\end{proof}

\begin{defn}\label{defn:6.9} $ $

 \begin{enumerate} \ealph \dispace
  \item We say that $q$ is \bfem{diagonally zero} if $q(\veps)=0$ for every $\veps\in\mfB.$

 \item We say that $q$ is \bfem{anisotropic} if $q(\veps)\ne0$ for every $\veps\in\mfB.$
     \end{enumerate} \ealph \dispace

\end{defn}

\begin{remarks}\label{remarks:6.10}
$ $

\begin{enumerate} \eroman  \dispace
\item
If $q$ is  diagonally zero, then $q$ is rigid,
cf.~\cite[Proposition 3.4]{QF1}. Conversely, if $q$ is rigid and the
quadratic form $[1]$ is quasilinear, i.e.,
$(\alpha+\beta)^2=\alpha^2+\beta^2$ for any $\alpha,\beta\in R,$
then $q$ is diagonally zero, as proved in \cite[Theorem 3.5]{QF1}.

\item  If $q$ is anisotropic, then $q(x)\ne0$ for every $x\in V\setminus\{0\}.$ So our definition of anisotropy here coincides with the usual meaning of anisotropy for quadratic forms (which makes sense, say, for $R$ a semiring without zero divisors and $V$ any $R$-module).
\end{enumerate}

\end{remarks}

\begin{defn}\label{defn:6.11}
In a similar vein, we call a symmetric \bfem{bilinear form} $b$ on $V$ \bfem{anisotropic} if $b(\veps,\veps)\ne0$ for every $\veps\in\mfB,$ and then have $b(x,x)\ne0$ for every $x\in V\setminus\{0\}$.
\end{defn}
\noindent  Note that, if $b$ is a balanced companion of $q$, then $b$ is anisotropic iff $q$ is anisotropic.

Assume now that $U:=(U,\gm )$ is a free bilinear module, $V:=(V,q)$ is a free quadratic module, and $b$ is a balanced companion of $q.$ Let
$$\tlV:=(\tlV,\tlq):=(U\otimes V,\gm \otimes_bq).$$
We want to determine the indecomposable components of $\tlV.$ Discarding trivial cases, we assume that $U\ne\langle 0\rangle ,$ $V\ne[0].$

We choose bases $\mfB_1$ and $\mfB_2$ of the $R$-modules $U$ and $V,$ respectively, and introduce the subsets
\begin{align*}
 \mfB_1^+&:=\{\veps\in\mfB_1 \ds|\gm (\veps,\veps)\ne0\},\\
 \mfB_1^0&:=\{\veps\in\mfB_1 \ds|\gm (\veps,\veps)=0\},\\
 \mfB_2^+&:=\{\eta\in\mfB_1 \ds|q(\eta)\ne0\},\\
 \mfB_2^0&:=\{\eta\in\mfB_1\ds|q(\eta)=0\},
\end{align*}
of $\mfB_1$ and $\mfB_2,$ respectively, and furthermore the basic
submodules $U^+,U^0,V^+,V^0$ respectively spanned by these sets.

\begin{lem}\label{lem:6.13}
\begin{enumerate} \dispace \item[a)] If $\veps\in\mfB_1^+,$ then the indecomposable components of the basic submodule $\veps\otimes V:=(R\veps)\otimes V$ of $U\otimes V$ with respect to $\tlq$ are the submodules $\veps\otimes W$ with $W$ running through the indecomposable components of $V$ with respect to $q.$

\item[b)] If $\eta\in\mfB_1^+,$ then the indecomposable components of $U\otimes \eta:=U\otimes (R\eta)$ with respect to~ $\tlq$ are the modules $U\otimes \eta$ with $U'$ running through the indecomposable components of $U$ with respect to the norm form $\eta(\gm )$ of $\gm $ (cf.~Definition~\ref{defn:5.12}).
    \end{enumerate}
    \end{lem}

    \begin{proof} This follows from the formulas $\tlq(\veps\otimes y)=\gm (\veps,\veps)q(y)$ for $y\in V$ and $\tlq(x\otimes\eta)=\gm (x,x)q(\eta)$ for $x\in U$ (cf.~\eqref{eq:5.4}), since $\gm (\veps,\veps)\ne0$, $q(\eta)\ne0.$ \end{proof}

    \begin{lem}\label{lem:6.14} Assume that $(V,q)$ is indecomposable. Let $a,c\in R\setminus\{0\}.$
    Then
    $$\begin{pmatrix} a& c\\ c &0\end{pmatrix}\otimes_bV=\begin{bmatrix}[c|c] aq &cb\\\hline &0\end{bmatrix}$$
    (cf.~\eqref{eq:5.10}) is indecomposable.
    \end{lem}

    \begin{proof} Let $$(U,\gm )=\begin{pmatrix} a &c\\ c& 0\end{pmatrix}$$ with respect to a base $\veps_1,\veps_2$  and assume for notational convenience that $V$ has a finite base $\eta_1,\dots,\eta_n.$ By \lemref{lem:6.13}.a, we have
    $$\veps_1\otimes\eta_1\sim\veps_1\otimes\eta_2\sim\cdots\sim\veps_1\otimes\eta_n.$$

    For given $\veps_1\otimes\eta_i,$ $\veps_2\otimes \eta_j$ with $i\ne j,$ $\gm \otimes_bq$ has the value table
    $$\begin{bmatrix} aq(\eta_i) & cb(\eta_i,\eta_j)\\
    & 0\end{bmatrix}.$$

Starting with $\veps_2\otimes\eta_j$, we find some $\eta_i,$ $i\ne j,$ with $b(\eta_i,\eta_j)\ne0,$ because $(V,q)$ is indecomposable. Since $R$ has \NQL, it follows that $R(\veps_i\otimes \eta_i)+R(\veps_j\otimes\eta_j)$ is indecomposable with respect to $\tlq,$ whence $\veps_1\otimes\eta_i\sim\veps_2\otimes\eta_j.$ Thus all $\veps_k\otimes\eta_\ell$ are equivalent.
\end{proof}

In order to avoid certain pathologies concerning indecomposability
in tensor products $U\otimes_bV,$ we henceforth will assume that our
semiring has the following property:
\begin{equation} \text{For any}\  a \ \text{and}\  c \ \text{in}\  R\setminus\{0\} \ \text{there exists some}\  \mu\in R \ \text{with} \ a+\mu c\ne a.  \tag{\NQL}
\end{equation}
Clearly, this property means that every free quadratic module
$\left[\begin{smallmatrix} a &c\\ &0\end{smallmatrix} \right]$ with
$c\ne0$ is \bfem{not quasilinear} on $(R\eta_1)\times (R\eta_2),$
where $(\eta_1,\eta_2)$ is the associated base, whence the label
``\NQL''.

\begin{examples}\label{examps:6.12}  $ $
\begin{enumerate} \ealph \dispace

\item In the important case that $R$ is supertropical
the condition (\NQL) holds iff all principal ideals in $eR$ are
unbounded with respect to the total ordering of $eR.$ In particular,
the ``multiplicatively unbounded supertropical semirings'' appearing
in \cite[\S4]{QF1} have \NQL.


 \item If $R$ is any entire antiring, then the polynomial ring $R[t]$ in one variable (and so in any set
of variables) has \NQL.

 \item The polynomial function semirings over supersemirings appearing in \cite[\S4]{IR1} have \NQL.

  \end{enumerate}
\end{examples}

\begin{lem}\label{lem:6.15} Assume that $(U,n(\gm ))$ is  indecomposable.
Let $a,c\in R\setminus\{0\}.$ Then the tensor product $U\otimes_b\begin{bmatrix} a& c\\ & 0\end{bmatrix},$ taken with respect to $b=\begin{pmatrix} 2a & c\\ c &0\end{pmatrix},$ is indecomposable.
\end{lem}

\begin{proof} By formula \eqref{eq:5.17}
$$\gm \otimes_b\begin{bmatrix}a & c\\ & 0\end{bmatrix}=\begin{pmatrix} a & c\\ c &0\end{pmatrix}\otimes_{2\gm }n(\gm ).$$
Now \lemref{lem:6.14} with $(V,q):=(U,n(\gm ))$ gives the claim.
\end{proof}

We are ready for the main result of this section. Recall that $U:=(U,\gm )$.

\begin{thm}\label{thm:6.16}
Assume that $R$ has \NQL. Assume furthermore that both $(U,n(\gm ))$
and the quadratic free module $V=(V,q)$ are indecomposable, and
$U\ne\langle0\rangle,$ $V\ne[0].$ Assume moreover that  is
indecomposable. Let $b$ be a balanced companion of $q.$ Then the
quadratic module $U\otimes_bV:=(U\otimes V,\gm \otimes_bq)$ is
indecomposable, except in the case that $\gm$ is alternate, $q$ is
diagonally zero, $U$ and $V$ contain only even cycles with respect
to $\gm $ and $b.$ Then $U\otimes_bV$ has exactly two indecomposable
components, and these coincide with the indecomposable components of
$U\otimes V$ with respect to $\gm \otimes b,$ and also with respect
to $\gm \otimes b_f.$
\end{thm}

\begin{proof} Of course, indecomposability
of $(U,n(\gm ))$ implies indecomposability of $(U,\gm ).$
As before, let $\tlq:=\gm \otimes_bq.$ We distinguish three cases.

1) Assume that $V^+\ne\{0\}$, i.e., there exist anisotropic base vectors in $V.$ Our claim is that all elements of $\mfB_1\otimes\mfB_2$ are equivalent, whence $U\otimes_bV$ is indecomposable.

We choose $\eta_0\in\mfB_2^+.$ By \lemref{lem:6.13}.b, the module
 $(U\otimes\eta_0,\tlq):=(U\otimes\eta_0,\tlq\ds|U\otimes\eta_0)$  is indecomposable, and thus all elements of $\mfB_1\otimes\eta_0$ are equivalent.

Let $\veps\otimes\eta\in\mfB_1\otimes\mfB_2.$ We verify the equivalence of $\veps\otimes\eta$ with some element of $\mfB_1\otimes\eta_0,$ and then will be done. If $\gm (\veps,\veps)\ne0$, then by \lemref{lem:6.13}.a, all elements of $\veps\otimes\mfB_2$ are equivalent, whence $\veps\otimes\eta\sim \veps\otimes\eta_0.$ Assume now that $\gm (\veps,\veps)=0.$ Since $(U,\gm )$ is indecomposable, there exists some $\veps'\in\mfB_1$ with $c:=\gm (\veps',\veps)\ne0.$ Let $a:=\gm (\veps',\veps').$ We choose a base $\eta_1,\dots,\eta_n$ of ~$V,$ assuming for notational convenience that $V$ has finite rank.
By Example \ref{examp:5.10},
$$(R\veps'+R\veps)\otimes_bV=\begin{bmatrix} aq & cq\\ & 0\end{bmatrix}$$
with respect to the base $\veps'\otimes\eta_1,\dots,\veps'\otimes\eta_n,\veps\otimes\eta_1,\dots,\veps\otimes\eta_2.$
Now \lemref{lem:6.14} tells us that $(R\veps'+R\veps)\otimes_bV$ is indecomposable, whence all elements $\veps\otimes\eta,$ $\veps'\otimes\eta'$ with $\eta,\eta'\in\mfB_2$ are equivalent. In particular, $\veps\otimes\eta\sim\veps'\otimes\eta_0.$ \pSkip

2) Assume that $U^+\ne\{0\}$, i.e., there exist an anisotropic base vector in $U$ with respect to $n(\gm ).$ Our claim again is  that all elements of $\mfB_1\otimes\mfB_2$ are equivalent, whence $U\otimes _bV$ is indecomposable. We choose $\veps_0\in\mfB_1^+,$ and then know by \lemref{lem:6.13}.a that all elements of $\veps_0\otimes\mfB_2$ are equivalent.

Let $\veps\otimes\eta\in\mfB_1\otimes\mfB_2$ be given. We verify equivalence of $\veps\otimes\eta$ with some element of $\veps_0\otimes\mfB_2,$ and then will be done. If $q(\eta)\ne0$, then by \lemref{lem:6.13}.a all elements of $\mfB_1\otimes\eta$ are equivalent, and thus $\veps\otimes\eta\sim\veps_0\otimes\eta.$

Hence, we may assume  that $q(\eta)=0.$ Since $(V,q)$ is indecomposable, there exists some $\eta'\in\mfB_2$ with $c:=b(\eta,\eta')\ne0.$ Let $a:=q(\eta').$ Then
$$(R\eta'+R\eta,q)=\begin{bmatrix} a &c\\ &0\end{bmatrix}.$$
Let $b':=b|(R\eta'+R\eta)=\begin{pmatrix} a &c\\ c &0\end{pmatrix}.$
Then we see from \eqref{eq:5.16} that
$$\gm \otimes_{b'}\begin{bmatrix} a & c\\ & 0\end{bmatrix}=\begin{bmatrix}[c|c] an(\gm ) & c\gm \\ \hline & 0\end{bmatrix}.$$
By \lemref{lem:6.15}, this quadratic module is indecomposable, whence all elements $\veps\otimes\eta,$ $\veps'\otimes\eta'$ with $\veps,\veps'\in\mfB_1$ are equivalent. In particular, $\veps\otimes\eta\sim\veps_0\otimes\eta'.$
\pSkip

3) The remaining case: $U=U^0,$ and $V=V^0,$ i.e., $\gm $ is alternate and $q$ is diagonally zero. Now $(U\otimes V,\tlq)$ is rigid. By \thmref{thm:6.8}, the indecomposable components of $(U\otimes V,\tlq)$ coincide with those of $(U\otimes V,(\gm \otimes b)_f).$ But $\tlq$ has only one companion, whence $(\gm \otimes b)_f=\gm \otimes b=\gm \otimes b_f.$ Invoking \thmref{thm:6.4}, we see that the assertion of the theorem also holds in the case under consideration, where $\gm $ is alternate and $b$ is diagonally zero.
\end{proof}

In general, let $\{U_i \ds|i\in I\}$ denote the set of indecomposable components of $(U,n(\gm ))$. Then
$$U\otimes_b V=\BigPerp_{i\in I}U_i\otimes_b V$$
by \propref{prop:5.6}, whence, applying \thmref{thm:6.16} to each summand $U_i\otimes _bV,$ we obtain a complete list of all indecomposable components of $U\otimes_b V.$ In particular, if $q$ is not diagonally zero, or if $(V,b)$ contains an odd cycle, then the $U_i\otimes_bV$ themselves are the indecomposable components of $U\otimes_bV.$

\end{document}